\documentclass[a4paper,11pt]{article}

\usepackage{amsfonts,qtree,stmaryrd,textcomp}
\usepackage{pifont,amssymb,amsmath,amsthm,tabularx,graphicx}
\usepackage{tree-dvips,pictexwd,dcpic}
\usepackage{bbm}
\usepackage{bm}

\usepackage{stmaryrd}

\usepackage[T1]{fontenc}
\usepackage[latin1]{inputenc}
\usepackage[text={17cm,26cm},centering]{geometry}

\usepackage{etex}
\usepackage{fancyhdr}
\usepackage{graphicx}
\usepackage{enumitem}
\usepackage{amsmath}
\usepackage[bbgreekl]{mathbbol}
\usepackage{amssymb}
\usepackage{amsthm}
\usepackage[all]{xy}

\usepackage{epigraph}
\RequirePackage{bussproofs}

\usepackage{mathabx}

\usepackage{scalerel,stackengine}

\usepackage{framed}
\usepackage{mathtools}
\usepackage{url}
\usepackage{proof}
\usepackage{xspace}
\usepackage{listings}
\usepackage{wrapfig}
\usepackage{tikz}
\usepackage{tikz-cd}

\usepackage{relsize}

\usetikzlibrary{positioning}
\usetikzlibrary{shapes}

\usetikzlibrary{arrows}
\usetikzlibrary{calc}
\usepackage{tikz-3dplot}
\usetikzlibrary{decorations.pathmorphing}

\usepackage{amsfonts,amssymb,amsmath,qtree,stmaryrd,textcomp, amsthm}

\usetikzlibrary{decorations.pathreplacing}
\tikzset{axis/.style={&lt;-&gt;}}

\stackMath
\newcommand\reallywidehat[1]{%
\savestack{\tmpbox}{\stretchto{%
  \scaleto{%
    \scalerel*[\widthof{\ensuremath{#1}}]{\kern-.6pt\bigwedge\kern-.6pt}%
    {\rule[-\textheight/2]{1ex}{\textheight}}
  }{\textheight}%
}{0.5ex}}%
\stackon[1pt]{#1}{\tmpbox}%
}
 \definecolor{MyBlue}{rgb}{0.05, 0.25, 0.65}
 
 \definecolor{MyRed}{rgb}{0.90, 0.05, 0.05}
 
\definecolor{MyGreen}{rgb}{0.05, 0.90, 0.05}

\newcommand{\C}[1]{\mathcal{#1}}
\newcommand{\D}[1]{\mathbb{#1}}

\usepackage{ mathrsfs }

\newtheorem{theorem}{Theorem}[section]
\newtheorem{proposition}[theorem]{Proposition}

\newtheorem{example}[theorem]{Example}

\newtheorem{definition}[theorem]{Definition}

\newcommand{\id}{\mathrm{id}}

\newcommand{\CST}{\mathrm{CST}}

\newcommand{\ZF}{\mathrm{ZF}}

\newcommand{\Cat}{\mathrm{\mathbf{Cat}}}

\newcommand{\TOT}{\Leftrightarrow}

\newcommand{\To}{\Rightarrow}

\newcommand{\sto}{\rightsquigarrow}

\newcommand{\MLTT}{\mathrm{MLTT}} 
\newcommand{\CZF}{\mathrm{CZF}}

\newcommand{\Fam}{\textnormal{\texttt{Fam}}}

\newcommand{\pr}{\textnormal{\texttt{pr}}}

\newcommand{\BST}{\mathrm{BST}}

\newcommand{\Set}{\mathrm{\mathbf{Set}}}

\newcommand{\op}{\mathrm{op}}

\newcommand{\Fun}{\mathrm{Fun}}

\newcommand{\Hom}{\mathrm{Hom}}

\newcommand{\dHom}{\textnormal{\texttt{dHom}}}

\newcommand{\fHom}{\textnormal{\texttt{fHom}}}

\newcommand{\f}{\textnormal{\texttt{fam}}}

\newcommand{\di}{\textnormal{\texttt{dep}}}

\newcommand{\si}{\textnormal{\texttt{s}}}

\newcommand{\Di}{\mathlarger{\mathlarger{\C D}}}

\newcommand{\cf}{\textnormal{\texttt{cofam}}}

\newcommand{\fba}{f \colon b \to a}

\newcommand{\BishSet}{\textnormal{\textbf{BishSet}}}

\newcommand{\Type}{\textnormal{\textbf{Type}}}

\newcommand{\ext}{\bigsqcup}

\newcommand{\wf}{\textnormal{\texttt{wfam}}}

\newcommand{\wfHom}{\textnormal{\texttt{wfHom}}}

\newcommand{\CaT}{\textnormal{CaT}}

\begin{document}

\date{}

\title{\textbf{Categories with dependent arrows}}

\author{Iosif Petrakis\\	
Department of Computer Science, University of Verona\\
iosif.petrakis@univr.it}  

%





\maketitle

\begin{abstract}
\noindent 
We present an abstract, categorical formulation of dependent functions in a fundamental manner and 
independently from the Sigma-construction. For that, 
we define first the notion of a category with family-arrows, or a $\f$-category. 
A $(\f, \Sigma)$-category is a $\f$-category with Sigma-objects, where a $(\f, \Sigma)$-category with a terminal object
is exactly a type-category of Pitts, or a category with attributes of Cartmell. We introduce categories with 
dependent arrows, or $\di$-categories, and we show that every $(\f, \Sigma)$-category is a $\di$-category in a 
canonical way. The notion of a Sigma-object in a $\di$-Category is affected by the existence of dependent arrows, 
and we show that every $(\f, \Sigma)$-category is a $(\di, \Sigma)$-category in a canonical way. \\[2mm]
\end{abstract}


\section{Introduction}
\label{sec: intro}

An important foundational difference between Zermelo-Fraenkel Set Theory $(\ZF)$ and category 
theory $(\CaT)$ is that in the latter the notion of (generalised) function i.e., of arrow is fundamental, while in the former
it is reduced to the concept of set. In this sense, $\CaT$ is much closer to Martin-L\"of Type Theory $(\MLTT)$
(see~\cite{ML75, ML84, ML98}) and Bishop Set Theory 
$(\BST)$ (see~\cite{Pe20, Pe22b}), which are theories of types (sets) \textit{and} functions. This feature of Bishop's 
theory of sets 
was captured by Myhill's formal system of Constructive Set Theory $\CST$ in~\cite{My75}, but it was not followed by Aczel in
his system of Constructive Zermelo-Fraenkel Set Theory $(\CZF)$ (see~\cite{AR10}). This similarity in the foundations
between $\CaT$ and $\MLTT$, or $\BST$, is in accordance with the extensive use of $\CaT$ in the semantics of type theories 
(e.g., see~\cite{LS86}).

A major feature, both of $\MLTT$, which originally was invented as a formal system for Bishop's book~\cite{Bi67}, and of $\BST$ 
is the use of dependent functions, or dependent assignment routines (see~\cite{Pe20}), as fundamental objects.
One may also say that even the notion of a type-family in $\MLTT$, or a set-indexed family of sets in $\BST$, is 
more or less fundamental. 
As it is noted by Palmgren in~\cite{Pa12a}, p.~35, in $\ZF$, and also in its constructive version $\CZF$, a 
family of sets is represented by the fibers of a function $\lambda \colon B \to I$, where the fibers 
$\lambda_i := \{ b \in B \mid \lambda(b) = i\}$ of $\lambda$, for every $i \in I$, represent the sets of the
family. Hence the notion of a family of sets is reduced to that of a set. As this reduction rests on the 
replacement scheme, such a reduction is not possible neither in $\MLTT$ nor in $\BST$.
We could say that the fundamental building-blocks of $\MLTT$ are the concepts 
\begin{center}
\textbf{types, functions, type-families, dependent functions,}
\end{center}
and that the building blocks of $\BST$ are the concepts
\begin{center}
\textbf{sets, functions, families of sets, dependent functions.}
\end{center}
The fundamental building blocks of $\CaT$ are the concepts
\begin{center}
\textbf{objects, arrows,}
\end{center}
and with the use of categorical notions, such as that of a \textbf{functor}, families of sets can be described as functors 
between categories. More abstract approaches to the notion of a family have been elaborated,
such as the notion of an
indexed-category (e.g., see~\cite{JP78}), where to every object of a category corresponds a certain category of families. 

The categorical interpretation of dependency has a long story (see e.g., the work of Cartmell~\cite{Ca78, Ca86}, 
Seely~\cite{Se84}, Ehrard~\cite{Eh88}, Curien~\cite{Cu93}, Dybjer~\cite{Dy96}, Jacobs~\cite{Ja99},
Hofmann~\cite{Ho97}, Pitts~\cite{Pi00}, and Palmgren~\cite{Pa12b}).
The dependent functions or the Pi-type (Pi-set) is translated categorically either as an object, or as a global section 
i.e., an arrow,
or even as an abstract family over an object, as in the framework of type-categories of Pitts~\cite{Pi00}. 
In the latter case, it depends on the
categorical interpretation of the Sigma-type.

Here we try to answer the following major question: \textit{What is the fundamental categorical generalisation of 
a family of sets and of a dependent function}? To answer this question, we incorporate appropriate, abstract 
formulations of the notions of a family of sets and of  a dependent function into the definition of a category.
Specifically, we introduce the notion of a \textit{category with family-arrows}, or a $\f$-\textit{category}, and the 
notion of a $\f$-\textit{category with dependent arrows}, or a $\di$-\textit{category}, so that the
fundamental concepts of the latter are
 \begin{center}
\textbf{objects, arrows, family-arrows, dependent arrows}.
\end{center}
The dependent-arrow-structure of such a category is axiomatised 
exactly as the arrow- and the family-structures.
Although dependency is captured categorically in many ways, here we propose to capture it as a primitive notion
and independently from the Sigma-type (set), exactly as it is the case in $\MLTT$ and $\BST$.

Adding a $\Sigma$-structure to a $\f$-category $\C C$, in a way compatible with the whole
$\f$-structure of $\C C$, results to the notion of an $(\f, \Sigma)$-category. $(\f, \Sigma)$-categories with
a terminal object are exactly the \textit{type-categories}, 
introduced by Pitts in~\cite{Pi00}, pp.~110-111,
following Cartmell's \textit{categories with attributes}, which were introduced in~\cite{Ca78}.
There are many examples of $(\f, \Sigma)$-categories without a terminal object (see Examples~\ref{ex: fconst}
and~\ref{ex: crings}). 
Pitts requires the existence of a terminal object, since his main example of a type-category in~\cite{Pi00} 
is that of the classifying category of 
a dependently typed algebraic theory, which has all finite products (see~\cite{Pi00}, p.~70). 
Notice that the notions of family-arrows and Sigma-objects are simultaneously given in the definition of 
a type-category by Pitts, while here are split. 
Here we also define $\di$-categories with Sigma-objects. As a $\di$-category has besides the arrow- and 
the family-structure a dependent-arrow-structure, the definition of a Sigma-object in it has to take into account 
the $\di$-arrow structure too.

We structure this paper as follows:
\vspace{-3mm}
\begin{itemize}
\item In section~\ref{sec: fcats} we present categories with family-arrows, and we give several examples of such categories.
To every object $a$ in such a category $\C C$ corresponds a collection of family-arrows $\fHom(a)$.

\item In section~\ref{sec: fscats} we present $\f$-categories with Sigma-objects, or $(\f, \Sigma)$-categories, together
with several examples of such categories.
Sigma-objects are abstract, categorical versions of the Sigma-types (sets) in $\MLTT$ $(\BST)$. We show that
in a $(\f, \Sigma)$-category $\C C$ with a terminal object $1$ we can recover the transport arrows
that witness the equality of 
the Sigma-objects over $1$ and the family-arrows $\lambda(i)$ and $\lambda(j)$, if $i, j$ are equal global 
elements of an object $a$ of $\C C$ (Proposition~\ref{prp: transp1}).

\item In section~\ref{sec: dcats} we introduce categories with dependent arrows, or $\di$-categories, together with many
examples of such categories. To every $a \in \C C$ and $\lambda \in \fHom(a)$ corresponds a collection $\dHom(a, \lambda)$ of 
\textit{dependent arrows} over $\lambda$. We show that a $(\f, \Sigma)$-category is a $\di$-category in a 
canonical way (Theorem~\ref{thm: typeisdi}).

\item In section~\ref{sec: dscats} we present $\di$-categories with sigma-objects, or $(\di, \Sigma)$-categories.
The second-projection-dependent arrow of a family-arrow $\lambda$ as an appropriate dependent
arrow is used, exactly as in $\MLTT$ $(\BST)$. We show that every $(\f, \Sigma)$-category is a $(\di, \Sigma)$-category
in a canonical way (Theorem~\ref{thm: typeisdsi}). Moreover, we show that if $\C C$ is a $(\di, \Sigma)$-category with 
a terminal object, then the two projection-arrows determine the corresponding Sigma-object (Proposition~\ref{prp: elsigma}).


\end{itemize}
\vspace{-3mm}


\noindent
For all notions and results from category theory that are used here without explanation or proof we
refer to~\cite{MM92, Aw10, Ri16}. 
$C_0$ denotes the objects of a category
$\C C$ and $C_1$ the arrows of $\C C$. $\Cat$ is the category of small categories.

\section{Categories with family-arrows}
\label{sec: fcats}

An arrow $f \colon a \to b$ in a category $\C C$ is the abstract, categorical version of a function $f \colon A \to B$,
and the standard categorical axioms for the composition of arrows are generalisations of the basic properties of
composition of functions.  
First, we extend the arrow-structure of a category with the abstract, categorical version of a family
of sets indexed by some set. For the sake of generality, an abstract family in $\C C$ over an object
$a \in \C C$ has no specific codomain, and it is composed in a coherent way with the arrows of $\C C$ with codomain $a$.

\begin{definition}\label{def: fcat}
A category $\C C$ is a category with family-arrows, or a $\f$-category, if\\[1mm]
\normalfont (i) 
\itshape For every object $a$ in $\C C$ there is a collection $\fHom(a)$, or $\fHom(a, \cdot)$, of
family-arrows. We denote the elements of $\fHom(a)$ by Greek letters $\lambda, \mu$, etc.  
If $\lambda \in \fHom(a)$,
we use a blue arrow starting from $a$, in order to picture $\lambda$
\vspace{-4mm}
\begin{center}
\begin{tikzpicture}

\node (E) at (0,0) {$a$};
\node[right=of E] (F) {.};

\draw[MyBlue,->] (E)--(F) node [midway,above] {$\lambda$};

\end{tikzpicture}
\end{center}
\vspace{-4mm}
$C_2 := \bigcup_{a \in C_0}\fHom(a)$ 
is the collection of 
family-arrows
of $\C C$.\\[1mm]
\normalfont (ii) 
\itshape For every $a, b \in \C C$ there is a composition-operation 
$\circ \colon \fHom(a) \times \Hom(b, a) \to \fHom(b)$, $(\lambda, f) \mapsto \lambda \circ f,$
such that 
the following compatibility conditions with the arrow-structure of $\C C$ hold:\\
$(\f_1)$ $ \ \lambda \circ 1_a = \lambda$
\vspace{-4mm}
\begin{center}
\begin{tikzpicture}

\node (E) at (0,0) {$a$};
\node[right=of E] (F) {$a$};
\node[right=of F] (A) {.};

\draw[->] (E)--(F) node [midway,above] {$1_a$};
\draw[MyBlue,->] (F)--(A) node [midway,above] {$\lambda$};
\draw[MyBlue,->,bend right=40] (E) to node [midway,below] {$\lambda$} (A) ;

\end{tikzpicture}
\end{center}
\vspace{-4mm}
$(\f_2)$ $ \ \lambda \circ (f \circ g) = (\lambda \circ f) \circ g$
\begin{center}
\resizebox{4cm}{!}{%
\begin{tikzpicture}

\node (E) at (0,0) {$c$};
\node[right=of E] (F) {$b$};
\node[right=of F] (A) {$a$};
\node[right=of A] (B) {.};

\draw[->] (E)--(F) node [midway,above] {$g$};
\draw[->] (F)--(A) node [midway,above] {$ \ f$};
\draw[MyBlue,->] (A)--(B) node [midway,above] {$\lambda$};
\draw[->,bend right] (E) to node [midway,below] {$f \circ g$} (A) ;
\draw[MyBlue,->,bend right=60] (E) to node [midway,below] {$\lambda \circ (f \circ g)$} (B) ;
\draw[MyBlue,->,bend left=40] (F) to node [midway,above] {$\lambda \circ f$} (B) ;
\draw[MyBlue,->,bend left=75] (E) to node [midway,above] {$(\lambda \circ f) \circ g$} (B) ;

\end{tikzpicture}
}
\end{center}
The family-structure of $\C C$ is called small, if $C_2$ is a set, and it is called locally small, 
if $\fHom(a)$ is a set, for every $a \in \C C$. If $C_2$ is a proper class, then we call $C_2$ large.
\end{definition}

\begin{example}[Families of sets and types]\label{ex: fsets}
 \normalfont Within the category of sets and functions $\Set$, if $I$ is a set, then a family of sets indexed by $I$ is a rule,
 or a functor, $\lambda \colon I \to \Set$, 
 in case $I$ is equipped with a trivial categorical structure (e.g., see~\cite{MRR88}, p.~18).
 Within the category $\BishSet$ of predicative sets in $\BST$, a family of sets over a set $I$ 
 is an appropriate non-dependent assignment routine
 $\lambda_0 \colon I \sto \D V_0$ (see~\cite{Pe20}). Within
 the category $\Type(\C U)$ of types in a universe $\C U$ of $\MLTT$,
 a family over a type $A \colon \C U$ is a term $P$ of type $A \to \C U$. The composition of family-arrows 
 with arrows is defined similarly 
 in the obvious way, in each case.
\end{example}

\begin{example}[Constant families]\label{ex: fconst}
 \normalfont A category $\C C$ is turned into a $\f$-category, if we define $\fHom(a) := C_0$, for every $a \in C_0$, and 
 $b \circ f = b$, for every $b \in C_0$ and $f \in \Hom(c, a)$
 \vspace{-3mm}
\begin{center}
\begin{tikzpicture}

\node (E) at (0,0) {$c$};
\node[right=of E] (F) {$a$};
\node[right=of F] (A) {.};

\draw[->] (E)--(F) node [midway,above] {$f$};
\draw[MyBlue,->] (F)--(A) node [midway,above] {$b$};
\draw[MyBlue,->,bend right=40] (E) to node [midway,below] {$b$} (A) ;

\end{tikzpicture}
\end{center}
\end{example}
\vspace{-3mm}

\begin{example}[The family-arrows in the coslice]\label{ex: fcoslice}
 \normalfont A category $\C C$ is turned into a $\f$-category, if we define $\fHom(a) := a/\C C$, for every $a \in C_0$, 
 where $a/\C C$ denotes the coslice of $\C C$ over $a$, and the composition is inherited from $\C C$.
\end{example}

\begin{example}[Families on categories]\label{ex: fcat}
\normalfont If $\C C$ is in $\Cat$, then we can define $\Fam(\C C) := 
\Fun(\C C^{\op}, \Set)$, the collection of all presheaves on $\C C$.
\end{example}

\begin{example}[Families in a topos (Pitts)]\label{ex: ftopos}
\normalfont If $\C C$ is a topos, 
with a subobject classifier $(\top, \Omega)$,
%
%
%
%
then, if $a \in \C C$, let a family on $a$ to be a pair $\lambda := (b, e)$, with $e \colon a \times b \to \Omega$ in $\C C$ 
i.e.,
$$\fHom(a) := \ext_{b \in C_0}\Hom(a \times b, \Omega).$$
If $g \colon c \to a$, let 
$(b, e) \circ g := \big(b, e \circ (g \times 1_b)\big)$ 
\vspace{-8mm}
\begin{center}
\resizebox{5cm}{!}{%
\begin{tikzpicture}

\node (E) at (0,0) {$a \ \ $};
\node[right=of E] (L) {};
\node[right=of L] (F) {$a \times \color{MyBlue} b$};
\node[right=of F] (N) {};
\node[right=of N] (A) {$ \ \ b$};
\node[above=of F] (K) {};
\node[above=of K] (B) {$c \times \color{MyBlue} b$};
\node[above=of L] (T) {$ \ c$};
\node[above=of N] (S) {$b  \ $};
\node[below=of F] (X) {$\Omega $};

\draw[->] (B)--(T) node [midway,left] {$\pr_c \ $};
\draw[->] (F)--(E) node [midway,below] {$\pr_a$};
\draw[->] (T)--(E) node [midway,left] {$g \ \ $};
\draw[->] (B)--(S) node [midway,right] {$ \ \pr_b$};
\draw[MyRed, ->,dashed] (B) to node [midway,left] {$g \times 1_b$} (F);
\draw[->] (F)--(A) node [midway,below] {$\pr_b$};
\draw[->] (S)--(A) node [midway,right] {$\ \ 1_b$};
\draw[MyBlue,->] (F)--(X) node [midway,left] {$e$};
\draw[MyBlue,->,bend left=175] (B) to node [midway,right] {$(b, e) \circ g$} (X) ;

\end{tikzpicture}
}
\end{center}
\end{example}
\vspace{-8mm}
In a weak version of a $\f$-category the ``strict'' conditions $(\f_1)$ and $(\f_2)$ hold up to
isomosphism. Next we give a fundamental example of a \textit{category with weak family-arrows}, or\footnote{There
is a debate over strict conditions vs weak conditions. Ehrhard advocates the weak concepts in~\cite{Eh88}, following
B\'enabou~\cite{Be85}, as more general approach and more categorical, since isomorphism is a ``more categorical'' 
concept than equality.
In~\cite{Pi00}, p.~113, Pitts defends the strict approach with respect to modeling dependent type theory. Our framework 
also indicates that $\f$-categories clearly correspond to categories and weak $\f$-categories correspond to 
weak categories.}
a $\wf$-\textit{category}.

\begin{example}[The weak family-arrows in the slice]\label{ex: fslice}
 \normalfont A category $\C C$ with pullbacks is turned into a $\wf$-category, if we define $\wfHom(a) := \C C/a$, for 
 every $a \in C_0$, 
 where $\C C/a$ denotes the slice of $\C C$ over $a$. If $\fba$ and $\lambda \colon c \to a$, the composition $\lambda \circ f$
 in $\C C/b$ is defined as the arrow $\lambda \circ f \colon b \times_a c \to b$ in the following 
 pullback\footnote{It is because of this example that Pitts in~\cite{Pi00}, pp.~110-111, calls the operation
 $\lambda \circ f$ in a type-category ``the pullback of  $\lambda$ along $a$''. See also his discussion on p.~113 on
 the use of this weak family-structure for the interpretation of dependent types in toposes, based on their 
 locally cartesian closed structure and the paradigm of Seely~\cite{Se84}.} 
  \begin{center}
\begin{tikzpicture}

\node (E) at (0,0) {$ b \ $};
\node[right=of E] (F) {$ \ \   a$.};
\node[above=of F] (A) {$c$};
\node [left=of A] (D) {$  \mathsmaller{b \times_a c} $};

\draw[->] (E) to node [midway,below] {$f$} (F) ;
\draw[->] (D) to node [midway,above] {$f{'}$} (A) ;
\draw[MyBlue,->] (D)--(E) node [midway,left] {$\lambda \circ f$};
\draw[MyBlue,->] (A)--(F) node [midway,right] {$\lambda$};

\end{tikzpicture}
\end{center}
As the following square is a pullback
\begin{center}
\begin{tikzpicture}

\node (E) at (0,0) {$ b $};
\node[right=of E] (F) {$a$};
\node[above=of F] (A) {$a$};
\node [left=of A] (D) {$b $};

\draw[->] (E) to node [midway,below] {$f$} (F) ;
\draw[->] (D) to node [midway,above] {$f$} (A) ;
\draw[->] (D)--(E) node [midway,left] {$1_b$};
\draw[->] (A)--(F) node [midway,right] {$1_a$};

\end{tikzpicture}
\end{center}
the arrows $\lambda \circ 1_a \colon b \times_a a \to b$ and $1_b$ are isomorphic in $\C C/b$.
The weak version of $(\f_2)$ is explained similarly.
\end{example}

If $\C C$ and $\C D$ are $\f$-categories, a $\f$-\textit{functor} $F \colon \C C \to \C D$ is a rule $F = (F_0, F_1, F_2)$, where 
$(F_0, F_1)$ is a functor and $F_2(\lambda) \in \fHom(F(a))$, for every $a \in \C C$ and 
$\lambda \in \fHom(a)$, and
$F_2(\lambda \circ f) = F_2(\lambda) \circ F_1(f)$, where $f \in \Hom(b, a)$ 
\begin{center}
\begin{tikzpicture}

\node (E) at (0,0) {$F(b)$};
\node[right=of E] (F) {$F(a)$};
\node[right=of F] (A) {.};

\draw[->] (E)--(F) node [midway,above] {$F(f)$};
\draw[MyBlue,->] (F)--(A) node [midway,above] {$F(\lambda)$};
\draw[MyBlue,->,bend right=40] (E) to node [midway,below] {$F(\lambda \circ f)$} (A) ;

\end{tikzpicture}
\end{center}
If $F, G \colon \C C \to \C D$ are $\f$-functors, a $\f$-\textit{natural transformation}
$\eta \colon F \To G$ is a natural transformation, such that, for every $a \in \C C$ and 
$\lambda \in \fHom(a)$, 
the following
triangle commutes
\begin{center}
\begin{tikzpicture}

\node (E) at (0,0) {$F(a)$};
\node[right=of E] (F) {};
\node[right=of F] (A) {$G(a)$};
\node[below=of F] (B) {.};

\draw[MyBlue,->] (E)--(B) node [midway,left] {$F(\lambda) \ $};
\draw[->] (E)--(A) node [midway,above] {$\eta_a$};
\draw[MyBlue,->] (A) to node [midway,right] {$\ G(\lambda)$} (B) ;

\end{tikzpicture}
\end{center}
$\f$-functors and $\f$-natural transformations are closed under composition, and the 
constant $\f$-functor is defined in the expected way. Many standard constructions, such as 
the product of $\f$-categories with projections as $\f$-functors, the slice $\f$-category and the 
coslice $\f$-category, are straightforward to develop.

Next we define the $\fHom$-functor for a $\f$-category $\C C$ with a locally small $\f$-structure. This presheaf
behaves similarly to the standard functor $\Hom(-, a)$ in a category. Using the corresponding category of elements, 
the category $\fHom(\C C)$ of family-arrows of $\C C$ is defined.

\begin{definition}\label{def: fcat}
If $\C C$ is a $\f$-category with a locally small $\f$-structure, let
$\fHom \colon \C C^{\op} \to \Set$, with $a \mapsto \fHom(a),$ and
$(f \colon b \to a) \mapsto \fHom(f) \colon \fHom(a) \to \fHom(b)$, where
$[\fHom(f)](\lambda) := \lambda \circ f.$
The category of family-arrows $\fHom(\C C)$, or $\C C_2$, of $\C C$ is the category\footnote{Here we follow
Palmgren's notation of the Grothendieck category found in~\cite{Pa16}. The connection of the Grothendieck
construction to the $\Sigma$-type of $\MLTT$ fully justifies Palmgren's notation.}
$\Sigma(\C C, \fHom)$ of elements of $\C C$ over the presheaf $\fHom$
i.e., $\fHom(\C C)$ has objects pairs $(a, \lambda)$ with $a \in C_0$ and $\lambda \in \fHom(a)$. An arrow $f \colon (b, \mu) \to 
(a, \lambda)$ is an arrow $f \colon b \to a$ in $C_1$ such that 
$\mu = [\fHom(f)](\lambda) := \lambda \circ f$
\begin{center}
\begin{tikzpicture}

\node (E) at (0,0) {$b$};
\node[right=of E] (F) {};
\node[right=of F] (A) {$a$};
\node[below=of F] (B) {.};

\draw[MyBlue,->] (E)--(B) node [midway,left] {$\mu \ $};
\draw[->] (E)--(A) node [midway,above] {$f$};
\draw[MyBlue,->] (A) to node [midway,right] {$\ \lambda$} (B) ;

\end{tikzpicture}
\end{center}
The composition of $f \colon (b, \mu) \to (a, \lambda)$ and $g \colon (c, \nu) \to (b, \mu)$ is $f \circ g$ and $1_{(a, \lambda)} :=
1_a$. 
\end{definition}

%
%

\section{Categories with family-arrows and Sigma-objects}
\label{sec: fscats}

Next we assign to each object $a$ of a $\f$-category $\C C$ and to each
$\lambda \in \fHom(a)$ a Sigma-object $\sum_a \lambda$ and its first-projection-arrow
$\pr_1^{a, \lambda} \colon \sum_a \lambda \to a$ in
$\C C$.

\begin{definition}\label{def: fscat}
A $\f$-category $\C C$ has Sigma-objects, or is a $(\f, \Sigma)$-category, if\\[1mm]
\normalfont (i) 
\itshape For every $a \in \C C$ there are operations
$$\sum_a \colon \fHom(a) \to C_0, \ \ \ \ \sum_a(\lambda) := \sum_a \lambda \in C_0,$$
$$\pr_1^a \colon \fHom(a) \to C_1, \ \ \ \ \pr_1^{a}(\lambda) := \pr_1^{a, \lambda} \colon \sum_a \lambda \to a,$$
where $\sum_a \lambda$ is the Sigma-object of $\lambda$, and $\pr_1^{a, \lambda}$ is the first-projection-arrow
in $\C C$ associated to the Sigma-object of $\lambda$.\\[1mm]
\normalfont (ii) 
\itshape For every $b \in \C C$ and $f \in \Hom(b, a)$ there is an operation
$$\Sigma f \colon \fHom(a) \to C_1, \ \ \ \ \lambda \mapsto (\Sigma f)(\lambda) 
=: \Sigma_{\lambda}f$$
in $\Hom\big(\sum_b (\lambda \circ f), \sum_a \lambda\big),$
such that the following square 
\begin{center}
\begin{tikzpicture}

\node (E) at (0,0) {$ b $};
\node[right=of E] (K) {};
\node[right=of K] (F) {$a$};
\node[above=of F] (A) {$ \ \  \sum_a \lambda$};
\node [left=of A] (D) {$\sum_b (\lambda \circ f) \   $};

\draw[->] (E) to node [midway,below] {$f$} (F) ;
\draw[->] (D) to node [midway,above] {$\Sigma_{\lambda}f$} (A) ;
\draw[->] (D)--(E) node [midway,left] {$\pr_1^{b, \lambda \circ f}$};
\draw[->] (A)--(F) node [midway,right] {$\pr_1^{a, \lambda}$};

\end{tikzpicture}
\end{center}
is a pullback, and the following strictness-conditions hold:\\[1mm]
$(\si_1)$ $\ \Sigma_{\lambda}1_a = 1_{\mathsmaller{ \sum_a \lambda}}$.\\[1mm]
$(\si_2)$ $\ \Sigma_{\lambda}(f \circ g) = \big(\Sigma_{\lambda}f\big) \circ \Sigma_{(\lambda \circ f)}g$, for every $f \in \Hom(b, a)$ 
and $g \in \Hom(c, b)$.
\end{definition}

Using conditions $(\f_1)$ and $(\f_2)$, conditions $(\si_1)$ and $(\si_2)$ are well-defined,
as the following rectangle is trivially a pullback  
\begin{center}
\resizebox{5cm}{!}{%
\begin{tikzpicture}

\node (E) at (0,0) {$ a $};
\node[right=of E] (K) {};
\node[right=of K] (F) {$a$};
\node[above=of F] (A) {$ \ \  \sum_a \lambda$};
\node [left=of A] (D) {$\sum_a (\lambda \circ 1_a)   \ $};

\draw[->] (E) to node [midway,below] {$1_a$} (F) ;
\draw[->] (D) to node [midway,above] {$\Sigma_{\lambda}1_a$} (A) ;
\draw[->] (D)--(E) node [midway,left] {$\pr_1^{a, \lambda \circ 1_a}$};
\draw[->] (A)--(F) node [midway,right] {$\pr_1^{a, \lambda}$};

\end{tikzpicture}
}
\end{center}
and by the pullback lemma the following outer rectangle is also a pullback 
\begin{center}
\resizebox{8cm}{!}{%
\begin{tikzpicture}

\node (E) at (0,0) {$\sum_c (\lambda \circ f) \circ g$};
\node[right=of E] (H) {};
\node[right=of H] (F) {$\sum_b (\lambda \circ f) $};
\node[below=of E] (A) {$c$};
\node[below=of F] (B) {$b$};
\node[right=of F] (K) {};
\node[right=of K] (G) {$\sum_a \lambda$};
\node [below=of G] (C) {$a$.};

\draw[->] (E)--(F) node [midway,above] {$ \mathsmaller{\Sigma_{(\lambda \circ f)}g}$};
\draw[->] (F)--(G) node [midway,above] {$ \mathsmaller{\Sigma_{\lambda}f}$};
\draw[->] (B)--(C) node [midway,below] {$ f $};
\draw[->] (A)--(B) node [midway,below] {$ \ \ g$};
\draw[->] (E)--(A) node [midway,left] {$\pr_1^{c, (\lambda \circ f) \circ g}$};
\draw[->] (F) to node [midway,left] {$\pr_1^{b, \lambda \circ f}$} (B);
\draw[->] (G)--(C) node [midway,right] {$\pr_1^{a, \lambda}$};
\draw[->,bend left] (E) to node [midway,above] {$\Sigma_{\lambda}(f \circ g)$} (G) ;

\end{tikzpicture}
}
\end{center}

\begin{example}[Trivial Sigma-object]\label{ex: strivial}
 \normalfont Every $\f$-category $\C C$  is turned into a $(\f, \Sigma)$-category. If $\lambda \in \fHom(a)$, let
 $\sum_a \lambda := a$, $\pr_1^{a, \lambda} := 1_a$, and $\Sigma_{\lambda}f = f$. The square 
 \begin{center}
\begin{tikzpicture}

\node (E) at (0,0) {$ b $};
\node[right=of E] (F) {$a$};
\node[above=of F] (A) {$a$};
\node [left=of A] (D) {$b $};

\draw[->] (E) to node [midway,below] {$f$} (F) ;
\draw[->] (D) to node [midway,above] {$f$} (A) ;
\draw[->] (D)--(E) node [midway,left] {$1_b$};
\draw[->] (A)--(F) node [midway,right] {$1_a$};

\end{tikzpicture}
\end{center}
is a pullback and conditions $(\si_1), (\si_2)$ are trivially satisfied.
\end{example}

\begin{example}[Sigma-set and Sigma-type]\label{ex: ssets}
 \normalfont In $\Set$ the Sigma-set of $I$ and $\lambda \colon I \to \Set$ is the exterior union 
 $$\ext_{i \in I}\lambda(i) := \bigg\{(i, x) \in I \times \bigcup_{i \in I} \lambda(i) \mid x \in \lambda(i)\bigg\},$$
 equipped with the
 corresponding projection-function to $I$.
 In the category $\BishSet$ of sets in $\BST$ the Sigma-set is the exterior union
 $\sum_{i \in I}\lambda_0 (i)$, the membership of which and its equality are defined by
\[ w \in \sum_{i \in I}\lambda_0 (i) : 
\TOT \exists_{i \in I}\exists_{x \in \lambda_0 (i)}\big(w := (i, x)\big), \]
\[ (i, x) =_{\mathsmaller{\sum_{i \in I}\lambda_0 (i)}} (j, y) : \TOT i =_I j \ \& \ \lambda_{ij} (x) 
=_{\lambda_0 (j)} y, \]
where $\lambda_{ij} \colon \lambda_0(i) \to \lambda_0(j)$ and 
$\lambda_{ji} \colon \lambda_0(j) \to \lambda_0(i)$ are the \textit{transport maps} that witness the equality 
of $\lambda_0(i)$ and $\lambda_0(j)$ in $\D V_0$ (see~\cite{Pe20}, p.~37). 
The assignment routine $\pr_1^{I, \lambda} \colon \sum_{i \in I}\lambda_0 (i) \sto I$, 
where $\pr_1^{I, \lambda} (i, x) := i$, for every $(i, x) \in \sum_{i \in I}\lambda_0 (i)$, is a function.
If $f \colon J \to I$, let 
$$\Sigma_{\lambda}f \colon \sum_{j \in J}\lambda_0(f(j)) \sto \sum_{i \in I}\lambda_0(i), \ \ \ \ (j, x) \mapsto (f(j), x).$$
Then $\Sigma_{\lambda}f$ is a function, the corresponding square is a pullback,
and conditions $(\si_1), (\si_2)$ are satisfied. For the category $\Type(\C U)$ we work similarly.
\end{example}

\begin{example}[Sigma-object of a constant family]\label{ex: sconst}
 \normalfont If $\C C$ has binary products, then it is turned into a $(\f, \Sigma)$-category as follows: if 
 $b \in \fHom(a)$, as in Example~\ref{ex: fconst}, we define
 $$\sum_a b := a \times b \ \ \& \ \ \pr_1^{a,b} := \pr_a \colon a \times b \to a.$$
 If $f \in \Hom(c, a)$, and if 
 $\Sigma_b f := \langle f \circ \pr_c, \pr_b\rangle =: f \times 1_b,$
%
%
%
then the following rectangle is a pullback
 \begin{center}
 \resizebox{6cm}{!}{%
\begin{tikzpicture}

\node (E) at (0,0) {$ c \ \ $};
\node[right=of E] (K) {};
\node[right=of K] (F) {$ \ \ a$};
\node[above=of F] (A) {$ \mathsmaller{a \times b}$};
\node [left=of A] (S) {};
\node[left=of S] (D) {$  \mathsmaller{c \times b}  $};
\node[left=of D] (G) {};
\node[above=of G] (H) {$d$};

\draw[->] (E) to node [midway,below] {$f$} (F) ;
\draw[->] (D) to node [midway,below] {$f \times 1_b$} (A) ;
\draw[->] (D)--(E) node [midway,left] {$\pr_b$};
\draw[->] (A)--(F) node [midway,right] {$\pr_a$};
\draw[->, bend right=50] (H) to node [midway,left] {$q$} (E);
\draw[->, bend left=30] (H) to node [midway,above] {$ \  p$} (A);
\draw[MyRed, ->,dashed] (H) to node [midway,right] {$\langle q, \pr_b \circ p\rangle$} (D) ;

\end{tikzpicture}
}
\end{center}
and conditions $(\si_1), (\si_2)$ are satisfied.
\end{example}

If $\C C$ has binary products, then working as in the previous example, we can define Sigma-objects 
over elements of $\C C$ and their coslices (see Example~\ref{ex: fcoslice}).
Clearly, to the family-structure in Example~\ref{ex: fcat} corresponds the Grothendieck construction.
To the family-structure of a topos in Example~\ref{ex: ftopos} Pitts corresponds in~\cite{Pi00}, p.~113,
a canonical construction of Sigma-objects. 
A weak version of Sigma-objects is defined in analogy to a category with weak family-arrows in Example~\ref{ex: fslice}.

\begin{example}[The weak Sigma-objects in the slice]\label{ex: fslice}
 \normalfont If $\C C$ is a category with pullbacks and $\wfHom(a) := \C C/a$, as in Example~\ref{ex: fslice},
 we define for every $\lambda \colon c \to a \in \C C/a$ the Sigma-object $\sum_a \lambda := c$,
 $\pr_1^{a, \lambda} := \lambda$, and if $\fba$, let $\Sigma_{\lambda} f := f{'}$, 
 which by definition is a pullback
 \begin{center}
\begin{tikzpicture}

\node (E) at (0,0) {$ b \  $};
\node[right=of E] (F) {$ \   a$.};
\node[above=of F] (A) {$c$};
\node [left=of A] (D) {$ \mathsmaller{b \times_a c} $};

\draw[->] (E) to node [midway,below] {$f$} (F) ;
\draw[->] (D) to node [midway,above] {$f{'}$} (A) ;
\draw[->] (D)--(E) node [midway,left] {$\lambda \circ f$};
\draw[->] (A)--(F) node [midway,right] {$\lambda$};

\end{tikzpicture}
\end{center}
The strictness conditions $(\si_1), (\si_2)$ are not satisfied, as the conditions $(\f_1), (\f_2)$ are not satisfied.
\end{example}

As we show next, there is a plethora of non-trivial $(\f, \Sigma)$-categories without a terminal object.

\begin{example}[Commutative rings]\label{ex: crings}
 \normalfont
If $(R, +, 0, \cdot, 1)$ is a commutative ring, and if $\C C(R, +, 0)$ is the category of its additive, group-structure
with objects a
singleton $\{\ast\}$ and arrows the elements of $R$, it is straightforward to show that every commutative square
 \begin{center}
\begin{tikzpicture}

\node (E) at (0,0) {$\ast $};
\node[right=of E] (F) {$\ast$};
\node[above=of F] (A) {$\ast$};
\node [left=of A] (D) {$\ast$};

\draw[->] (E) to node [midway,below] {$c$} (F) ;
\draw[->] (D) to node [midway,above] {$a$} (A) ;
\draw[->] (D)--(E) node [midway,left] {$d$};
\draw[->] (A)--(F) node [midway,right] {$b$};

\end{tikzpicture}
\end{center}
is a pullback. If $\Fam(\ast) := R \times R$ and $(a, b) \circ c := (c + a, c +b)$, for every $a, b, c \in R$, we equip 
$\C C(R, +, 0)$ with a family-arrow-structure. If we define $\sum_{\ast}(a, b) := \ast$, $\pr_1^{\ast, (a, b)} := a \cdot b$,
and $\Sigma_{(a,b)}c := c(1+c+b+a)$, 
 \begin{center}
\begin{tikzpicture}

\node (E) at (0,0) {$\ast $};
\node[right=of E] (K) {};
\node[right=of K] (F) {$\ast$};
\node[above=of F] (A) {$\ast$};
\node [above=of E] (D) {$\ast$};

\draw[->] (E) to node [midway,below] {$c$} (F) ;
\draw[->] (D) to node [midway,above] {$c(1+c+b+a)$} (A) ;
\draw[->] (D)--(E) node [midway,left] {$(c+a)\cdot(c+b)$};
\draw[->] (A)--(F) node [midway,right] {$a \cdot b$};

\end{tikzpicture}
\end{center}
we turn $\C C(R, +, 0)$ into a $(\f, \Sigma)$-category, which, in general, has no terminal object.
\end{example}

A notion of a $(\f, \Sigma)$-\textit{functor} can be defined in the expected way, and one can show that a ring
homomorphism between two
commutative rings induces a $(\f, \Sigma)$-functor between the corresponding $(\f, \Sigma)$-categories.

If $\C C$ is a $(\f, \Sigma)$-category with a terminal object $1$, we can recover within $\C C$ 
the transport maps $\lambda_{ij} \colon \lambda_0(i) \to \lambda_0(j)$, where $i =_I j$, from the definition of
an $I$-family of sets in $\BST$ (see also Example~\ref{ex: ssets}). Clearly, $\lambda_0(i)$ and $\sum_{k \in 1}\mu_0(k)$,
where $\mu_0(0) := \lambda_0(i)$, are equal in $\D V_0$.

\begin{proposition}\label{prp: transp1} If $\C C$ is a $(\f, \Sigma)$-category with a terminal object $1$, 
$a \in \C C$ and $i, j \in a$, the following hold:
\begin{center}
\resizebox{3cm}{!}{%
\begin{tikzpicture}

\node (E) at (0,0) {$1$};
\node[right=of E] (F) {$a$};
\node[right=of F] (A) {.};

\draw[>->,bend right=25] (E) to node [midway,below] {$i$} (F);
\draw[MyBlue,->,bend right=65] (E) to node [midway,below] {$\lambda(i)$} (A) ;
\draw[>->,bend left=25] (E) to node [midway,above] {$j$} (F);
\draw[MyBlue,->,bend left=65] (E) to node [midway,above] {$\lambda(j)$} (A) ;
\draw[MyBlue,->] (F) to node [midway,above] {$\lambda$} (A) ;

\end{tikzpicture}
}
\end{center}
\normalfont (i) 
\itshape $\sum_1 \lambda(i)$ is a subobject of $\sum_a \lambda$, and $\pr_1^{1, \lambda(i)} =!$, the unique arrow from 
$\sum_1 \lambda(i) \to 1$
\begin{center}
\resizebox{5cm}{!}{%
\begin{tikzpicture}

\node (E) at (0,0) {$\sum_1 \lambda(i) $};
\node[right=of E] (K) {};
\node[right=of K] (F) {$\sum_a \lambda$};
\node[below=of E] (A) {$1 $};
\node [below=of F] (D) {$a$.};

\draw[>->] (E) to node [midway,above] {$\Sigma_{\lambda}i$} (F) ;
\draw[>->] (A) to node [midway,below] {$i$} (D) ;
\draw[->] (E)--(A) node [midway,left] {$!$};
\draw[->] (F)--(D) node [midway,right] {$\pr_1^{a, \lambda}$};

\end{tikzpicture}
}
\end{center}
\normalfont (ii) 
\itshape If $i = j$, there are transport arrows $\lambda_{ij} \colon \sum_1 \lambda(i) \to \sum_1 \lambda(j)$ and 
$\lambda_{ji} \colon \sum_1 \lambda(j) \to \sum_1 \lambda(i)$, which form an iso.
\end{proposition}

\begin{proof}
 (i) It follows from the basic property of pullbacks, as $i$ is a mono, and hence so is $\Sigma_{\lambda}i$.\\
 (ii) As the following square is a pullback, and as the outer diagram commutes, exactly for the same reason, there is unique 
 arrow $\lambda_{ji}$, such that $\Sigma_{\lambda}i \circ \lambda_{ji} = \Sigma_{\lambda}j$ and $! \circ \lambda_{ji} = !!$.
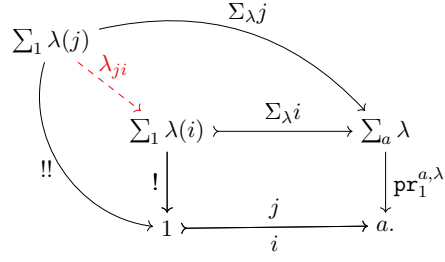
 \begin{center}
 \resizebox{6cm}{!}{%
\begin{tikzpicture}

\node (E) at (0,0) {$\sum_1 \lambda(i) $};
\node[right=of E] (K) {};
\node[right=of K] (F) {$\sum_a \lambda$};
\node[below=of E] (A) {$1 $};
\node [below=of F] (D) {$a$.};
\node[left=of E] (G) {};
\node[above=of G] (H) {$\sum_1 \lambda(j)$};

\draw[>->] (E) to node [midway,above] {$\Sigma_{\lambda}i$} (F) ;
\draw[>->] (A) to node [midway,below] {$i$} (D) ;
\draw[>->] (A) to node [midway,above] {$j$} (D) ;
\draw[->] (E)--(A) node [midway,left] {$!$};
\draw[->] (F)--(D) node [midway,right] {$\pr_1^{a, \lambda}$};
\draw[->, bend right=50] (H) to node [midway,left] {$!!$} (A);
\draw[->, bend left=30] (H) to node [midway,above] {$ \  \Sigma_{\lambda}j$} (F);
\draw[MyRed, ->,dashed] (H) to node [midway,above] {$ \  \lambda_{ji}$} (E) ;
\draw[->] (E)--(A) node [midway,left] {$!$};

\end{tikzpicture}
}
\end{center}
The following dual pullback determines the arrow $\lambda_{ij}$, satisfying 
$\Sigma_{\lambda}j \circ \lambda_{ij} = \Sigma_{\lambda}i$ and $!! \circ \lambda_{ij} = !$.
 
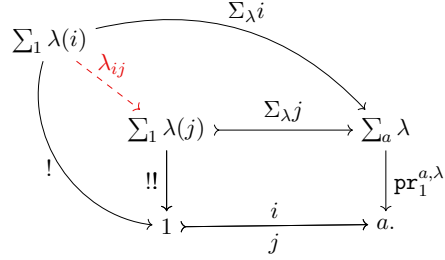
 \begin{center}
 \resizebox{6cm}{!}{%
\begin{tikzpicture}

\node (E) at (0,0) {$\sum_1 \lambda(j) $};
\node[right=of E] (K) {};
\node[right=of K] (F) {$\sum_a \lambda$};
\node[below=of E] (A) {$1 $};
\node [below=of F] (D) {$a$.};
\node[left=of E] (G) {};
\node[above=of G] (H) {$\sum_1 \lambda(i)$};

\draw[>->] (E) to node [midway,above] {$\Sigma_{\lambda}j$} (F) ;
\draw[>->] (A) to node [midway,below] {$j$} (D) ;
\draw[>->] (A) to node [midway,above] {$i$} (D) ;
\draw[->] (E)--(A) node [midway,left] {$!$};
\draw[->] (F)--(D) node [midway,right] {$\pr_1^{a, \lambda}$};
\draw[->, bend right=50] (H) to node [midway,left] {$!$} (A);
\draw[->, bend left=30] (H) to node [midway,above] {$ \  \Sigma_{\lambda}i$} (F);
\draw[MyRed, ->,dashed] (H) to node [midway,above] {$ \  \lambda_{ij}$} (E) ;
\draw[->] (E)--(A) node [midway,left] {$!!$};

\end{tikzpicture}
}
\end{center}
By the equalities
$\Sigma_{\lambda}i \circ 1_{\sum_1 \lambda(i)} = \Sigma_{\lambda}i
 = \Sigma_{\lambda}j \circ \lambda_{ij} 
= \big(\Sigma_{\lambda}i \circ \lambda_{ji}\big)  \circ \lambda_{ij} 
 = \Sigma_{\lambda}i \circ \big(\lambda_{ji}  \circ \lambda_{ij}\big)$
and since $\Sigma_{\lambda}i$ is a mono, we get $1_{\sum_1 \lambda(i)} = \lambda_{ji}  \circ \lambda_{ij}$.
Working similarly, we get $1_{\sum_1 \lambda(j)} = \lambda_{ij}  \circ \lambda_{ji}$.
\end{proof}

\section{Categories  with dependent arrows}
\label{sec: dcats}

Next, we extend the arrow-structure and the family-structure of a $\f$-category with the abstract, categorical version 
of a dependent function. To every $a \in \C C$ and $\lambda \in \fHom(a)$ corresponds a collection $\dHom(a, \lambda)$ of 
\textit{dependent arrows} over $a$ and $\lambda$. In the presence of dependent arrows in a $\di$-category $\C C$, 
its standard arrows can also be 
called the \textit{non-dependent arrows} of $\C C$. The axioms of a $\di$-category ensure the compatibility of 
the dependent-arrow structure with the (non-dependent) arrow- and family-structure of the given $\f$-category.

\begin{definition}\label{def: dcat}
A $\f$-category $\C C$ has dependent arrows, or is a $\di$-category, if\\[1mm]
\normalfont (i) 
\itshape For every object $a$ in $\C C$ and $\lambda \in \fHom(a)$ there is a collection $\dHom(a, \lambda)$
of dependent arrows over $(a, \lambda)$. We denote the elements of $\dHom(a, \lambda)$ by capital Greek
letters $\Phi, \Psi$, etc.
Let $C_3 := \bigcup_{a \in C_0, \lambda \in \fHom(a)}\dHom(a, \lambda)$ be the collection of all dependent-arrows of
$\C C$.\\[1mm]
\normalfont (ii) 
\itshape For every $\Phi \in \dHom(a, \lambda)$ and every $f \in \Hom(b, a)$ there is a dependent arrow 
$\Phi(f) \in \dHom(b, \lambda \circ f)$, which we call the application\footnote{One could call it the composition
$\Phi \circ f$, instead of the application of $\Phi$ to  $f$, but there are examples of $\di$-categories in
which there is already a notion of composition between the dependent arrows and the non-dependent ones e.g.,
see the global sections or dependent objects in Theorem~\ref{thm: typeisdi}.} of $\Phi$ to $f$, such that 
the following compatibility conditions with the $\f$-structure of $\C C$ hold:\\
$(\di_1)$ $\ \Phi(1_a) = \Phi$.\\[1mm]
$(\di_2)$ $\ \Phi(f \circ g) = [\Phi(f)](g)$, where $f \in \Hom(b, a)$ and $g \in \Hom(c, b)$.\\[1mm]
We call dependent-arrow-structure of $\C C$ small, if $C_3$ is a set, and locally small, 
if $\dHom(a, \lambda)$ is a set, for every $a \in C_0$ and $\lambda \in \fHom(a)$.
If $C_3$ is a proper class, we call $C_3$ large.
\end{definition}

Using conditions $(\f_1)$ and $(\f_2)$ we have that $(\di_1)$ and $(\di_2)$ are well-defined,
as $\Phi(1_a) \in \fHom(a, \lambda \circ 1_a)$, $\Phi(f \circ g) \in \dHom(c, \lambda \circ (f \circ g))$,
and $[\Phi(f)](g) \in \dHom(c, (\lambda \circ f) \circ g)$.
The notion of a dependent arrow is a categorical generalisation of the notion of
dependent function in $\MLTT$ or $\BST$, exactly as the notion of arrow is the categorical generalisation of the notion 
of function. The most fundamental feature of a dependent function $\Phi$ with respect to a family $(\lambda(i))_{i \in I}$ 
of types (sets) over a type (set) $I$ is that if $i \colon I$ $(i \in I)$, then $\Phi(i) \colon \lambda(i)$
$(\Phi(i) \in \lambda(i))$. If $\C C$ is a $\di$-category 
with a terminal object $1$, then 
\begin{center}
\begin{tikzpicture}

\node (E) at (0,0) {$1$};
\node[right=of E] (F) {$a$};
\node[right=of F] (A) {};

\draw[>->] (E)--(F) node [midway,above] {$i$};
\draw[MyBlue,->] (F)--(A) node [midway,above] {$\lambda$};
\draw[MyBlue,->,bend right=40] (E) to node [midway,below] {$\lambda(i)$} (A) ;

\end{tikzpicture}
\end{center}
if $i \in a$ and $\Phi \in \dHom(a, \lambda)$, then $\Phi(i) \in \dHom(1, \lambda(i))$.

\begin{example}[Trivial dependent arrows]\label{ex: dtrivial}
 \normalfont Every $\f$-category $\C C$  is turned into a $(\di)$-category. For every $a \in \C C$ and $\lambda \in \fHom(a)$
 let  $\dHom(a, \lambda) := \{\ast\}$. Then 
conditions $(\di_1)$ and $(\di_2)$ are trivially satisfied.
\end{example}

\begin{example}[Dependent arrows in sets and types]\label{ex: dsets}
 \normalfont In $\Set$, if $\lambda \colon I \to \Set$ is a family of sets over $I$, its dependent arrows are the 
 elements of the product set  
 $$\dHom(I, \lambda) := \prod_{i \in I}\lambda(i) := $$
 $$\bigg\{x \colon I \to \bigcup_{i \in I} \lambda(i) \mid
 \forall_{i \in I}\big(x_i := x(i)
 \in \lambda(i)\big)\bigg\}$$
 \begin{center}
\begin{tikzpicture}

\node (E) at (0,0) {$I$};
\node[right=of E] (F) {$ \bigcup_{i \in I} \lambda(i)$};
\node[right=of F] (A) {$I$.};

\draw[->] (E)--(F) node [midway,above] {$x$};
\draw[->] (F)--(A) node [midway,above] {$ \pr_i$};
\draw[->,bend right=30] (E) to node [midway,below] {$\id_X$} (A) ;

\end{tikzpicture}
\end{center}
 If $f \colon J \to I$ and $x \in \prod_{i \in I}\lambda(i)$, let 
 $$x \circ f \colon J \to \bigcup_{j \in J} \lambda(f(j)), \ \ 
(x \circ f)_j := x_{f(j)}, \ \ j \in J.$$
Clearly, conditions $(\di_1)$ and $(\di_2)$ are satisfied.
 The Pi-type in the category of types $\Type(\C U)$
 and the Pi-set in the category $\BishSet$ 
 behaves similarly (see~\cite{Pe20}, p.~47).
\end{example}

\begin{example}[Alternative dependent arrows in $\BishSet$]\label{ex: altdsets}
\normalfont One could have taken as family-arrows on a set $I$ the assignment routines $\lambda_0 \colon 
I \sto \D V_0$ without using the transport maps, and as dependent arrows over $I$ and $\lambda_0$ one could have considered 
the (fundamental) dependent assignment routines that just output an element of
$\lambda_0(i)$ for every given $i \in I$ (see~\cite{Pe20}, pp.~15-16).
\end{example}

\begin{example}[Dependent arrows of constant families]\label{ex: dconst}
\normalfont Any category $\C C$ is turned into a $\di$-category, if we define $\fHom(a) := C_0$, as in Example~\ref{ex: fconst},
and
$$\dHom(a, b) := \Hom(a, b) \ \ \& $$
$$f(g) := f \circ g \in \dHom(c, b \circ g) := \dHom(c, b) := \Hom(c, b),$$
for every $f \in \Hom(a, b)$ and $g \in \Hom(c, a)$.
\end{example}

Next we show that any $(\f, \Sigma)$-category, hence any type-category, is turned into a $\di$-category, in a canonical way.
For that we consider what Pitts calls in~\cite{Pi00}, p.~114, a \textit{global section}, or what 
Hofmann and Streicher call a \textit{dependent object} in~\cite{HS98}, pp.~91-92. The use of their category to represent
the Pi-category in~\cite{Pe22a} is a special case of Theorem~\ref{thm: typeisdi}. 
The arrow $\phi(f)$ defined next is noticed by Pitts, but here we highlight its special role in the following proof.

\begin{theorem}\label{thm: typeisdi}
 If $\C C$ is an $(\f, \Sigma)$-category, let for every $a \in \C C$ and $\lambda \in \dHom(a)$ 
 $$\Di_{a}\lambda := \bigg\{\phi \in \Hom\bigg(a, \sum_a \lambda\bigg) \mid \pr_1^{a, \lambda} \circ \phi = 1_a\bigg\}$$
\begin{center}
\begin{tikzpicture}

\node (E) at (0,0) {$a$};
\node[right=of E] (F) {$ \sum_a \lambda$};
\node[right=of F] (A) {$a$};

\draw[->] (E)--(F) node [midway,above] {$\phi$};
\draw[->] (F)--(A) node [midway,above] {$ \pr_1^{a, \lambda}$};
\draw[->,bend right=30] (E) to node [midway,below] {$1_a$} (A) ;

\end{tikzpicture}
\end{center}
be the  set of dependent objects of $\lambda$. With the dependent structure $\dHom(a, \lambda) := \Di_a \lambda$
the  $(\f, \Sigma)$-category $\C C$ becomes a $\di$-category.
\end{theorem}

\begin{proof}
If $\phi \in \Di_{a}\lambda$ and $f \in \Hom(b, a)$ we define a global section $\phi(f) \in \Di_{b}(\lambda \circ f)$ 
as follows
\begin{center}
\resizebox{5cm}{!}{%
\begin{tikzpicture}

\node (E) at (0,0) {$b$};
\node[right=of E] (F) {$ \sum_b (\lambda \circ f)$};
\node[right=of F] (K) {};
\node[right=of K] (A) {$b.$};

\draw[->] (E)--(F) node [midway,above] {$\phi(f)$};
\draw[->] (F)--(A) node [midway,above] {$ \pr_1^{b, \lambda \circ f}$};
\draw[->,bend right=40] (E) to node [midway,below] {$1_b$} (A) ;

\end{tikzpicture}
}
\end{center}
As $\pr_1^{a, \lambda} \circ (\phi \circ f) = (\pr_1^{a, \lambda} \circ \phi) = 1_a \circ f = f = f \circ 1_b$,
\begin{center}
\resizebox{5cm}{!}{%
\begin{tikzpicture}

\node (E) at (0,0) {$ b $};
\node[right=of E] (K) {};
\node[right=of K] (F) {$a$};
\node[above=of F] (A) {$ \  \sum_a \lambda$};
\node[left=of A] (D) {$\sum_b (\lambda \circ f) \ \ \ $};
\node[left=of D] (G) {};
\node[above=of G] (H) {$b$};

\draw[->] (E) to node [midway,below] {$f$} (F) ;
\draw[->] (D) to node [midway,above] {$\Sigma_{\lambda}f$} (A) ;
\draw[->] (D)--(E) node [midway,left] {$\pr_1^{b, \lambda \circ f}$};
\draw[->] (A)--(F) node [midway,right] {$\pr_1^{a, \lambda}$};
\draw[->, bend right=50] (H) to node [midway,left] {$1_b$} (E);
\draw[->, bend left=30] (H) to node [midway,above] {$ \  \phi \circ f$} (A);
\draw[MyRed, ->,dashed] (H) to node [midway,above] {$ \ \ \ \phi(f)$} (D) ;

\end{tikzpicture}
}
\end{center}
the above outer diagram commutes, and as the above square is a pullback, let $\phi(f)$ be the unique arrow in 
$\Hom\big(b, \sum_b (\lambda \circ f)\big)$ that makes the above triangles commutative i.e., 
\begin{equation}\label{eq: g1}
 \phi \circ f = \big(\Sigma_{\lambda}f\big) \circ \phi(f),
\end{equation}
\begin{equation}\label{eq: g2}
 \pr_1^{b, \lambda \circ f} \circ \phi(f) = 1_b.
\end{equation}
Condition $(\di_1)$ follows from~(\ref{eq: g1}) and condition $(\si_1)$, since
$$\phi = \phi \circ 1_a = \big(\Sigma_{\lambda}1_a\big) \circ \phi(1_a) = 1_{\Sigma_a \lambda}  \circ \phi(1_a) = \phi(1_a).$$
If $g \in \Hom(c, b)$, then $\phi(f \circ g)$ is the unique arrow in $\Hom\big(c, \sum_c (\lambda \circ f 
\circ g)\big)$ such that 
\begin{center}
\resizebox{8cm}{!}{%
\begin{tikzpicture}

\node (E) at (0,0) {$\sum_c (\lambda \circ f) \circ g$};
\node[right=of E] (H) {};
\node[right=of H] (F) {$\sum_b (\lambda \circ f) $};
\node[below=of E] (A) {$c$};
\node[below=of F] (B) {$b$};
\node[right=of F] (K) {};
\node[right=of K] (G) {$\sum_a \lambda$};
\node [below=of G] (C) {$a$};
\node[left=of E] (S) {};
\node[above=of S] (T) {$c$};
\node[above=of E] (M) {};
\node[right=of M] (X) {$b$};

\draw[->] (E)--(F) node [midway,above] {$ \mathsmaller{\Sigma_{(\lambda \circ f)}g}$};
\draw[->] (F)--(G) node [midway,above] {$ \mathsmaller{\Sigma_{\lambda}f}$};
\draw[->] (B)--(C) node [midway,below] {$ f $};
\draw[->] (A)--(B) node [midway,below] {$ \ \ g$};
\draw[->] (E)--(A) node [midway,left] {$\pr_1^{c, (\lambda \circ f) \circ g}$};
\draw[->] (F) to node [midway,left] {$\pr_1^{b, \lambda \circ f}$} (B);
\draw[->] (G)--(C) node [midway,right] {$\pr_1^{a, \lambda}$};

\draw[->, bend right=65] (T) to node [midway,left] {$1_c$} (A);
\draw[MyRed, ->,dashed] (T) to node [midway,right] {$ \ \   \phi(f \circ g)$} (E) ;
\draw[->, bend left=25] (T) to node [midway,above] {$g$} (X);
\draw[->, bend left=25] (X) to node [midway,above] {$\phi \circ f$} (G);
\draw[MyRed, ->,dashed] (X) to node [midway,right] {$ \ \ \phi(f)$} (F) ;

\end{tikzpicture}
}
\end{center}
\begin{equation}\label{eq: g3}
 \phi \circ f \circ g = \big(\Sigma_{\lambda}(f \circ g)\big) \circ \phi(f \circ g),
\end{equation}
\begin{equation}\label{eq: g4}
 \pr_1^{c, \lambda \circ f \circ g} \circ \phi(f \circ g) = 1_c.
\end{equation}
Moreover, $[\phi(f)](g)$ is the unique arrow in $\Hom\big(c, \sum_c (\lambda \circ f 
\circ g)\big)$ such that 
\begin{center}
\resizebox{5cm}{!}{%
\begin{tikzpicture}

\node (E) at (0,0) {$ c $};
\node[right=of E] (K) {};
\node[right=of K] (F) {$  b$};
\node[above=of F] (A) {$  \  \mathsmaller{\sum_b (\lambda \circ f)}$};
\node[left=of A] (D) {$\mathsmaller{\sum_c (\lambda \circ f) \circ g} \ \ $};
\node[left=of D] (G) {};
\node[above=of G] (H) {$c$};

\draw[->] (E) to node [midway,below] {$g$} (F) ;
\draw[->] (D) to node [midway,above] {$\Sigma_{(\lambda \circ f)}g$} (A) ;
\draw[->] (D)--(E) node [midway,left] {$\pr_1^{c, \lambda \circ f \circ g}$};
\draw[->] (A)--(F) node [midway,right] {$\pr_1^{b, \lambda \circ f}$};
\draw[->, bend right=60] (H) to node [midway,left] {$1_c$} (E);
\draw[->, bend left=30] (H) to node [midway,above] {$ \ \ \  \phi(f) \circ g$} (A);
\draw[MyRed, ->,dashed] (H) to node [midway,above] {$ \ \ \ \ \ \ \ \ \ \ \ \ [\phi(f)](g)$} (D) ;

\end{tikzpicture}
}
\end{center}
\begin{equation}\label{eq: g5}
 \phi(f) \circ g = \Sigma_{(\lambda \circ f)}g \circ [\phi(f)](g),
\end{equation}
\begin{equation}\label{eq: g6}
 \pr_1^{c, \lambda \circ f \circ g} \circ [\phi(f)](g) = 1_c.
\end{equation}
In order to show condition $(\di_2)$, it suffices to show that $\phi(f \circ g)$ satisfies the last two equalities. Due 
to~(\ref{eq: g4}) we have that $\phi(f \circ g)$ satisfies~(\ref{eq: g6}). In order to show that it also satisfies~(\ref{eq: g5})
i.e., that the rectangle consisting of the two red arrows in the diagram above commutes, we use that the following rectangle
\begin{center}
\resizebox{8cm}{!}{%
\begin{tikzpicture}

\node (E) at (0,0) {$ b $};
\node[right=of E] (K) {};
\node[right=of K] (F) {$a$};
\node[above=of F] (A) {$ \ \  \sum_a \lambda$};
\node[left=of A] (D) {$\sum_b (\lambda \circ f) \ \  $};
\node[left=of D] (G) {};
\node[above=of G] (H) {$\ \ b$};
\node[above=of H] (L) {};
\node[left=of L] (U) {};
\node[left=of U] (S) {$c \ \  $};
\node[right=of H] (X) {};
\node[right=of X] (T) {$b$};
\node[below=of S] (W) {$ \ \ \ \mathsmaller{\sum_c (\lambda \circ f) \circ g}  $};

\draw[->] (E) to node [midway,below] {$f$} (F) ;
\draw[->] (D) to node [midway,above] {$\Sigma_{\lambda}f$} (A) ;
\draw[->] (D)--(E) node [midway,left] {$\pr_1^{b, \lambda \circ f}$};
\draw[->] (A)--(F) node [midway,right] {$\pr_1^{a, \lambda}$};
\draw[->, bend right=100] (S) to node [midway,left] {$g \  $} (E);
\draw[->, bend left=25] (T) to node [midway,right] {$ \ \phi \circ f$} (A);
\draw[MyRed, ->,dashed] (H) to node [midway,above] {$ \ \ \phi(f)$} (D) ;
\draw[->] (S)--(H) node [midway,right] {$ \ g$};
\draw[->, bend left=25] (S) to node [midway,right] {$\ \  g$} (T);
\draw[MyRed, ->,dashed] (S) to node [midway,right] {$  \mathsmaller{\phi(f \circ g)} $} (W);
\draw[->, bend right=20] (W) to node [midway,left] {$ \mathsmaller{\Sigma_{(\lambda \circ f)}g} \ \  $} (D);
\end{tikzpicture}
}
\end{center}
\vspace{-13mm}
is by Definition~\ref{def: fscat} a pullback, and one shows from the previous commutativities that both these arrows from $c$ to
$\sum_b (\lambda \circ f)$ make the corresponding left and right upper diagrams commutative, hence they are equal.
\end{proof}

\begin{example}[Dependent objects of constant families]\label{ex: exdo2}
\normalfont If $\C C$ has products, then by Example~\ref{ex: sconst} $\C C$ is an $(\f, \Sigma)$-category and if $b \in
\dHom(a)$, its collection of dependent objects is
$$\Di_{a} b := \big\{\phi \in \Hom(a, a \times b) \mid \pr_a \circ \phi = 1_a\big\}.$$
By Example~\ref{ex: dconst} the canonical $\di$-structure on $\C C$ is determined by the equality
$\dHom(a, b) := \Hom(a, b)$. The two $\di$-structures can be identified, as there is a bijection 
$$e \colon \Hom(a, b) \to \Di_{a} b,  \ \ \ e(f) := \langle 1_a, f \rangle,$$
$$j \colon \Di_{a} b \to \Hom(a, b), \ \ \ j(\phi) := \pr_b \circ \phi.$$
 \end{example}

%
%
%
%
%
%
%
%

\begin{definition}\label{def: dcat}
If $\C C$ is an $\di$-category with a locally small $\di$-structure, let
$$\dHom \colon \fHom(\C C)^{\op} \to \Set, \ \ \ \ (a, \lambda) \mapsto \dHom(a, \lambda),$$
$$[f \colon (b, \mu) \to (a, \lambda)] \mapsto \dHom(f) \colon \dHom(a, \lambda) \to \dHom(b, \mu), $$
$$[\dHom(f)](\Phi) := \Phi(f).$$ 
The category of dependent-arrows $\dHom(\C C)$, or $\C C_3$,  of $\C C$ is the category $\Sigma\big(\fHom(\C C), \dHom\big)$
of elements of $\fHom(\C C)$ over  the presheaf $\dHom$
i.e., $\dHom(\C C)$ has objects pairs $\big((a, \lambda), \Phi\big)$ with $a \in C_0$, $\lambda \in \fHom(a)$, and
$\Phi \in \dHom(a, \lambda)$. An arrow $f \colon \big((b, \mu), \Psi\big) \to 
\big((a, \lambda), \Phi\big)$ is an arrow $f \colon (b, \mu) \to (a, \lambda)$ in $\fHom(\C C)$, such that 
$\Psi = [\dHom(f)](\Phi) := \Phi(f)$.
\end{definition}

\section{Categories with dependent arrows and Sigma-objects}
\label{sec: dscats}

Next we assign to each $a$ in a $\di$-category $\C C$ and to each
$\lambda \in \dHom(a)$ a Sigma-object $\sum_a \lambda$, its first-projection-arrow
$\pr_1^{a, \lambda} \colon \sum_a \lambda \to a$ in
$C_1$, and also its second-projection-dependent arrow
$\pr_2^{a, \lambda} \in \dHom\big(\sum_a \lambda, \lambda \circ \pr_1^{a, \lambda}\big)$ in $C_3$.

\begin{definition}\label{def: dscat}
A $\di$-category $\C C$ has Sigma-objects, or is a $(\di, \Sigma)$-category, if\\[1mm]
\normalfont (i) 
\itshape For every $a $ and $\fba$ in $\C C$, there are operations
$$\sum_a \colon \fHom(a) \to C_0, \ \ \ \ \sum_a(\lambda) := \sum_a \lambda \in C_0,$$
$$\pr_1^a \colon \fHom(a) \to C_1, \ \ \ \ \pr_1^{a}(\lambda) := \pr_1^{a, \lambda} \colon \sum_a \lambda \to a,$$
$$\Sigma f \colon \fHom(a) \to C_1, \ \ \ \ \lambda \mapsto (\Sigma f)(\lambda) 
=: \Sigma_{\lambda}f$$
with which $\C C$ becomes an $(\f, \Sigma)$-category.\\[1mm]
\normalfont (ii) 
\itshape If $a \in \C C$, there is an operation $\pr_2^a \colon \fHom(a) \to C_3$
$$\pr_2^{a}(\lambda) := \pr_2^{a, \lambda} \in 
\dHom\big(\sum_a \lambda, \lambda \circ \pr_1^{a, \lambda}\big),$$
\begin{center}
\begin{tikzpicture}

\node (E) at (0,0) {$\sum_a \lambda$};
\node[right=of E] (K) {};
\node[right=of K] (F) {$a$};
\node[right=of F] (A) {.};

\draw[->] (E)--(F) node [midway,above] {$\pr_1^{a, \lambda}$};
\draw[MyBlue,->] (F)--(A) node [midway,above] {$\lambda$};
\draw[MyBlue,->,bend right=25] (E) to node [midway,below] {$\lambda \circ \pr_1^{a, \lambda}$} (A) ;

\end{tikzpicture}
\end{center}
where $\pr_2^{a, \lambda}$ is called the second-projection-dependent arrow of $\lambda$, such that 
for every $b \in \C C$ and $f \in \Hom(b, a)$ the following condition is satisfied
$$\pr_2^{b, \lambda \circ f} = \pr_2^{a, \lambda}\big(\Sigma_{\lambda}f\big).$$
\end{definition}

Notice that the equality in condition (ii) above is well-defined, as by definition we have that $\pr_2^{b, \lambda \circ f}$ is in 
$$\dHom\bigg(\sum_b (\lambda \circ f), (\lambda \circ f) \circ \pr_1^{b, \lambda \circ f}\bigg)  =$$
$$\dHom\bigg(\sum_b (\lambda \circ f), \lambda \circ (\pr_1^{a, \lambda} \circ \Sigma_{\lambda}f)\bigg)$$
\begin{center}
\begin{tikzpicture}

\node (E) at (0,0) {$ b $};
\node[right=of E] (K) {};
\node[right=of K] (F) {$a$};
\node[above=of F] (A) {$ \ \  \sum_a \lambda$};
\node [left=of A] (D) {$\sum_b (\lambda \circ f) \  \  $};
\node [right=of F] (G) {};

\draw[->] (E) to node [midway,below] {$f$} (F) ;
\draw[->] (D) to node [midway,above] {$\Sigma_{\lambda}f$} (A) ;
\draw[->] (D)--(E) node [midway,left] {$\pr_1^{b, \lambda \circ f}$};
\draw[->] (A)--(F) node [midway,right] {$\pr_1^{a, \lambda}$};
\draw[MyBlue,->] (F)--(G) node [midway,below] {$\lambda \ $};
\draw[MyBlue,->,bend right=45] (E) to node [midway,right] {} (G) ;

\end{tikzpicture}
\end{center}
and by Definition~\ref{def: dcat} we have that 
$$\pr_2^{a, \lambda}\big(\Sigma_{\lambda}f\big) \in \dHom\bigg(\sum_b (\lambda \circ f), 
(\lambda \circ \pr_1^{a, \lambda}) \circ \Sigma_{\lambda}f\bigg).$$

\begin{example}[Trivial projection-arrows]\label{ex: exdstrivial}
 \normalfont A $\f$-category $\C C$  is turned into a $(\di, \Sigma)$-category. Using the dependent-structure of
 Example~\ref{ex: dtrivial}, let $\pr_2^{a, \lambda} := \ast$.
\end{example}

Clearly, $\Type(\C U)$ and $\BishSet$ are  $(\di, \Sigma)$-categories, where the second-projection-arrow is defined 
in each case in the obvious way.

\begin{example}[Sigma-objects of constant families]\label{ex: exds2}
 \normalfont If $\C C$ has binary products, then it is turned into a $(\di, \Sigma)$-category as follows: by 
 Example~\ref{ex: sconst} $\C C$ is an $(\f, \Sigma)$-category, while if
 $$\pr_2^{a, b} := \pr_b \in \dHom(a \times b, b \circ \pr_a) :=$$
 $$\dHom(a \times b, b) := \Hom(a \times b, b),$$
 then by the definition of $f \times 1_b$ we get 
 $$\pr_2^{a, b}\big(\Sigma_b f\big) := \pr_2\big(\Sigma_b f\big) := \pr_b \circ (f \times 1_b) =$$
 $$\pr_b =: \pr_2^{a, b} = 
 \pr_2^{a, b \circ f}.$$
\end{example}

Extending Theorem~\ref{thm: typeisdi}, we have that a $(\f, \Sigma)$-category is a $(\di, \Sigma)$-category in
a canonical way.

\begin{theorem}\label{thm: typeisdsi}
 If $\C C$ is an $(\f, \Sigma)$-category, let for every $a \in \C C$ and $\lambda \in \dHom(a)$ the dependent arrow
 $$\pr_2^{a, \lambda} \in \Di_{\sum_a \lambda}\big(\lambda \circ \pr_1^{a, \lambda}\big) = $$
 $$\mathsmaller{\mathsmaller{
 \bigg\{\phi \in \Hom\big(\sum_a \lambda, \sum_{\sum_a \lambda}(\lambda \circ \pr_1^{a, \lambda})\big) 
 \mid \pr_1^{\sum_a \lambda, \lambda \circ \pr_1^{a, \lambda}} \circ \pr_2^{a, \lambda} = 1_{\sum_a \lambda}\bigg\}}}$$
\begin{center}
\resizebox{7cm}{!}{%
\begin{tikzpicture}

\node (E) at (0,0) {$\sum_a \lambda$};
\node[right=of E] (K) {};
\node[right=of K] (F) {$ \sum_{\sum_a \lambda} (\lambda \circ \pr_1^{a, \lambda})$};
\node[right=of F] (L) {};
\node[right=of L] (A) {$\sum_a \lambda$};

\draw[->] (E)--(F) node [midway,above] {$\pr_2^{a, \lambda} $};
\draw[->] (F)--(A) node [midway,above] {$ \pr_1^{\sum_a \lambda, \lambda \circ \pr_1^{a, \lambda}}$};
\draw[->,bend right=25] (E) to node [midway,below] {$1_{\sum_a \lambda}$} (A) ;

\end{tikzpicture}
}
\end{center}
defined as the unique arrow determined by the following pullback
\begin{center}
\resizebox{8cm}{!}{%
\begin{tikzpicture}

\node (E) at (0,0) {$  \sum_a \lambda$};
\node[right=of E] (K) {};
\node[right=of K] (F) {$a$};
\node[above=of F] (A) {$ \ \ \ \sum_a \lambda$};
\node[left=of A] (D) {$\mathsmaller{\sum_{\sum_a \lambda} (\lambda \circ \pr_1^{a, \lambda})} \ \ \ $};
\node[left=of D] (G) {};
\node[above=of G] (H) {$\sum_a \lambda$};

\draw[->] (E) to node [midway,below] {$\pr_1^{a, \lambda}$} (F) ;
\draw[->] (D) to node [midway,above] {$\mathsmaller{\Sigma_{\lambda}\pr_1^{a, \lambda}}$} (A) ;
\draw[->] (D)--(E) node [midway,left] {$\mathsmaller{\pr_1^{\sum_a \lambda, \lambda \circ \pr_1^{a, \lambda}}}$};
\draw[->] (A)--(F) node [midway,right] {$\pr_1^{a, \lambda}$};
\draw[->, bend right=60] (H) to node [midway,left] {$1_{\sum_a \lambda}$} (E);
\draw[->, bend left=30] (H) to node [midway,above] {$ \  1_{\sum_a \lambda}$} (A);
\draw[MyRed, ->,dashed] (H) to node [midway,above] {$ \ \  \ \pr_2^{a, \lambda}$} (D) ;

\end{tikzpicture}
}
\end{center}
Then $\C C$ becomes a $(\di, \Sigma)$-category.
\end{theorem}

\begin{proof}
By Theorem~\ref{thm: typeisdi} we have that $\C C$ is a $\di$-category. By the commutativity of the upper,
left triangle we have that $\pr_2^{a, \lambda}$ is in $\Di_{\sum_a \lambda}\big(\lambda \circ \pr_1^{a, \lambda}\big)$.
Hence, it suffices only to show that the dependent arrow 
$\pr_2^{a, \lambda}$ defined above satisfies the equality in condition (ii) of Definition~\ref{def: dscat}. 
If $f \in \Hom(b, a)$,
then by definition the second-projection-dependent arrow $\pr_2^{b, \lambda \circ f}$ is uniquely determined by the following 
left pullback
\begin{center}
\resizebox{9cm}{!}{%
\begin{tikzpicture}

\node (E) at (0,0) {$\mathsmaller{\sum_{\sum_b (\lambda \circ f)} [(\lambda \circ f) \circ \pr_1^{b, \lambda \circ f}]}$};
\node[right=of E] (H) {};
\node[right=of H] (F) {$\mathsmaller{\sum_b (\lambda \circ f)} $};
\node[below=of E] (A) {$\mathsmaller{\mathsmaller{\sum_b (\lambda \circ f)}}$};
\node[below=of F] (B) {$b$};
\node[right=of F] (K) {};
\node[right=of K] (G) {$\mathsmaller{\sum_a \lambda}$};
\node [below=of G] (C) {$a$};
\node[left=of E] (S) {};
\node[above=of S] (T) {$\mathsmaller{\sum_b (\lambda \circ f)}$};
\node[above=of E] (M) {};

\draw[->] (E)--(F) node [midway,above] {$ \mathsmaller{\Sigma_{(\lambda \circ f)}\pr_1^{b, \lambda \circ f}}$};
\draw[->] (F)--(G) node [midway,above] {$ \mathsmaller{\Sigma_{\lambda}f}$};
\draw[->] (B)--(C) node [midway,below] {$ f $};
\draw[->] (A)--(B) node [midway,below] {$ \ \ \pr_1^{b, \lambda \circ f}$};
\draw[->] (E)--(A) node [midway,right] {$\mathsmaller{\pr_1^{\sum_b (\lambda \circ f), (\lambda \circ f) \circ \pr_1^{b, 
\lambda \circ f}}}$};
\draw[->] (F) to node [midway,right] {$\mathsmaller{\pr_1^{b, \lambda \circ f}}$} (B);
\draw[->] (G)--(C) node [midway,left] {$\mathsmaller{\pr_1^{a, \lambda}}$};

\draw[->, bend right=45] (T) to node [midway,left] {$\mathsmaller{1_{\sum_b (\lambda \circ f)}}$} (A);
\draw[MyRed, ->,dashed] (T) to node [midway,right] {$ \ \   \pr_2^{b, \lambda \circ f}$} (E) ;
\draw[->, bend left=25] (T) to node [midway,below] {$\mathsmaller{1_{\sum_b (\lambda \circ f)}}$} (F);

\end{tikzpicture}
}
\end{center}
and by the definition of $\phi(f)$ in the proof of Theorem~\ref{thm: typeisdi} the arrow
$\pr_2^{a, \lambda}\big(\Sigma_{\lambda}f\big)$ is uniquely determined by the following left pullback
\begin{center}
\resizebox{9cm}{!}{%
\begin{tikzpicture}

\node (E) at (0,0) {$\mathsmaller{\sum_{\sum_b (\lambda \circ f)} [(\lambda \circ \pr_1^{a, \lambda}) \circ \Sigma_{\lambda} f]}$};
\node[right=of E] (H) {};
\node[right=of H] (F) {$\mathsmaller{\sum_{\sum_a \lambda} (\lambda \circ \pr_1^{a, \lambda})}$};
\node[below=of E] (A) {$\sum_b (\lambda \circ f)$};
\node[below=of F] (B) {$\mathsmaller{\sum_a \lambda}$};
\node[right=of F] (K) {};
\node[right=of K] (G) {$\sum_a \lambda$};
\node [below=of G] (C) {$a$};
\node[left=of E] (S) {};
\node[above=of S] (T) {$\sum_b (\lambda \circ f)$};
\node[above=of E] (M) {};

\draw[->] (E)--(F) node [midway,above] {$ \mathsmaller{\Sigma_{\big(\lambda \circ \pr_1^{a, \lambda}\big)}(\Sigma_{\lambda}f)}$};
\draw[->] (F)--(G) node [midway,above] {$ \mathsmaller{\Sigma_{\lambda}\pr_1^{a, \lambda}}$};
\draw[->] (B)--(C) node [midway,below] {$\pr_1^{a, \lambda}$};
\draw[->] (A)--(B) node [midway,below] {$\Sigma_{\lambda}f$};
\draw[->] (E)--(A) node [midway,right] {$\mathsmaller{\pr_1^{\sum_b (\lambda \circ f), (\lambda \circ f) \circ \pr_1^{b, \lambda \circ f}}}$};
\draw[->] (F) to node [midway,right] {$\mathsmaller{\pr_1^{\sum_a \lambda, \lambda \circ \pr_1^{a, \lambda}}}$} (B);

\draw[->] (G)--(C) node [midway,left] {$\mathsmaller{\pr_1^{a, \lambda}}$};

\draw[->, bend right=45] (T) to node [midway,left] {$1_{\sum_b (\lambda \circ f)}$} (A);

\draw[MyRed, ->,dashed] (T) to node [midway,right] {$ \ \  \pr_2^{a, \lambda}\big(\Sigma_{\lambda}f\big)$} (E) ;
\draw[->, bend left=25] (T) to node [midway,above] {$\pr_2^{a, \lambda} \circ \Sigma_{\lambda} f$} (F);

\draw[->, bend left=40] (T) to node [midway,above] {$\Sigma_\lambda f$} (G);

\end{tikzpicture}
}
\end{center}
By the commutativity of the right rectangle of the first diagram above we have that 
$$\sum_{\sum_b (\lambda \circ f)} \big[(\lambda \circ f) \circ \pr_1^{b, \lambda \circ f}\big] =
\sum_{\sum_b (\lambda \circ f)} \big[(\lambda \circ \pr_1^{a, \lambda}) \circ \Sigma_{\lambda} f\big]$$
and the composition of the below outer arrows in both big diagrams above are equal. By condition $(\si_2)$ we have that the 
 composition of the upper outer arrows in both big diagrams above are equal, as
$$\big(\Sigma_{\lambda} f\big) \circ \Sigma_{(\lambda \circ f)}\pr_1^{b, \lambda \circ f} = \Sigma_{\lambda}\big(f \circ 
\pr_1^{b, \lambda \circ f}\big)$$
and
$$\big(\Sigma_{\lambda}\pr_1^{a, \lambda}\big) \circ \Sigma_{\big(\lambda \circ \pr_1^{a, \lambda}\big)}(\Sigma_{\lambda}f)
= $$
$$\Sigma_{\lambda} \big(\pr_1^{a, \lambda} \circ \Sigma_{\lambda} f\big) = \Sigma_{\lambda}\big(f \circ 
\pr_1^{b, \lambda \circ f}\big).$$
Consequently, the two arrows $\pr_2^{b, \lambda \circ f}$ and $\pr_2^{a, \lambda}\big(\Sigma_{\lambda}f\big)$ are equal,
as by the pullback lemma the outer diagrams above, which are equal, are also pullbacks and as by the definition of 
$\pr_2^{a, \lambda}$ 
we have that
\begin{center}
\resizebox{8cm}{!}{%
\begin{tikzpicture}

\node (E) at (0,0) {$\sum_b (\lambda \circ f) \ $};
\node[right=of E] (K) {};
\node[right=of K] (L) {};
\node[right=of L] (F) {$a$};
\node[above=of F] (A) {$ \   \sum_a \lambda$};
\node [left=of A] (S) {};
\node[left=of S] (D) {$\mathsmaller{\sum_{\sum_b (\lambda \circ f)} [(\lambda \circ f) \circ \pr_1^{b, \lambda \circ f}]} \  \ \ $};
\node[left=of D] (G) {};
\node[above=of G] (H) {$\sum_b (\lambda \circ f)$};

\draw[->] (E) to node [midway,below] {$f \circ \pr_1^{b, \lambda \circ f}$} (F) ;
\draw[->] (D) to node [midway,above] {$\mathsmaller{\Sigma_{\lambda}\big(f \circ 
\pr_1^{b, \lambda \circ f}\big)}$} (A) ;
\draw[->] (D)--(E) node [midway,right] {$\mathsmaller{\pr_1^{\sum_b (\lambda \circ f), (\lambda \circ f) \circ \pr_1^{b, \lambda \circ f}}}$};
\draw[->] (A)--(F) node [midway,right] {$\pr_1^{a, \lambda}$};
\draw[->, bend right=45] (H) to node [midway,left] {$ 1_{\sum_b (\lambda \circ f)}$} (E);
\draw[->, bend left=30] (H) to node [midway,above] {$ \  \Sigma_{\lambda} f$} (A);
\draw[MyRed, ->,dashed] (H) to node [midway,above] {$ \ \  \ \ \pr_2^{b, \lambda \circ f}$} (D) ;
\draw[MyRed, ->,dashed] (H) to node [midway,below] {$\mathsmaller{\pr_2^{a, \lambda}(\Sigma_{\lambda}f)} \ \ \ \ \ $} (D) ;

\end{tikzpicture}
}
\end{center}
$$\Sigma_{\lambda}\pr_1^{a, \lambda} \circ \pr_2^{a, \lambda} \circ \Sigma_{\lambda} f = 1_{\sum_a \lambda}
\circ \Sigma_{\lambda} f =  \Sigma_{\lambda} f,$$
and by the uniqueness property of the last pullback, we get the required equality.  
\end{proof}

If $\C C$ is $(\di, \Sigma)$-category, one can relate the second-projection arrow generated by the previous
theorem from its $(\f, \Sigma)$-structure to its given one.

\begin{example}[Dependent objects of constant families]\label{ex: expr22}
 \normalfont If $\C C$ has binary products, then it is turned into a $(\di, \Sigma)$-category in Example~\ref{ex: exds2}.
 As by Example~\ref{ex: sconst} $\C C$ is a $(\f, \Sigma)$-category, its induced second-projection-dependent arrow from
 Theorem~\ref{thm: typeisdsi} satisfies 
 \begin{center}
 \resizebox{7cm}{!}{%
\begin{tikzpicture}

\node (E) at (0,0) {$\sum_a b$};
\node[right=of E] (K) {};
\node[right=of K] (F) {$ \sum_{\sum_a b} (b \circ \pr_1^{a, b})$};
\node[right=of F] (L) {};
\node[right=of L] (A) {$\sum_a b$};

\draw[->] (E)--(F) node [midway,above] {$\pr{'}_2^{a, b} $};
\draw[->] (F)--(A) node [midway,above] {$ \pr_1^{\sum_a b, b \circ \pr_1^{a, b}}$};
\draw[->,bend right=25] (E) to node [midway,below] {$1_{\sum_a b}$} (A) ;

\end{tikzpicture}
}
\end{center}
i.e.,
\begin{center}
\resizebox{7cm}{!}{%
\begin{tikzpicture}

\node (E) at (0,0) {$a \times b$};
\node[right=of E] (K) {};
\node[right=of K] (F) {$ (a \times b) \times b$};
\node[right=of F] (L) {};
\node[right=of L] (A) {$a \times b$};

\draw[->] (E)--(F) node [midway,above] {$\pr{'}_2^{a, b} $};
\draw[->] (F)--(A) node [midway,above] {$ \pr_{a \times b}$};
\draw[->,bend right=25] (E) to node [midway,below] {$1_{a \times b}$} (A) ;

\end{tikzpicture}
}
\end{center}
and it is defined as the unique arrow determined by the following pullback
\begin{center}
\resizebox{6cm}{!}{%
\begin{tikzpicture}

\node (E) at (0,0) {$  a \times b \ \ \ \ $};
\node[right=of E] (K) {};
\node[right=of K] (F) {$ \ \ \ a$};
\node[above=of F] (A) {$ \mathsmaller{a \times b} $};
\node [left=of A] (S) {};
\node[left=of S] (D) {$ \ \mathsmaller{(a \times b) \times b}  $};
\node[left=of D] (G) {};
\node[above=of G] (H) {$a \times b$};

\draw[->] (E) to node [midway,below] {$\pr_a$} (F) ;
\draw[->] (D) to node [midway,above] {$\mathsmaller{\pr_a \times 1_b}$} (A) ;
\draw[->] (D)--(E) node [midway,left] {$\pr_{a \times b}$};
\draw[->] (A)--(F) node [midway,right] {$\pr_a$};
\draw[->, bend right=60] (H) to node [midway,left] {$1_{a \times b}$} (E);
\draw[->, bend left=30] (H) to node [midway,above] {$ \  1_{a \times b}$} (A);
\draw[MyRed, ->,dashed] (H) to node [midway,above] {$ \ \ \ \ \ \ \pr{'}_2^{a, b}$} (D) ;

\end{tikzpicture}
}
\end{center}
As one can show that $\pr{'}_2^{a, b} = 1_{a \times b} \times \pr_b$, one can identify $\pr{'}_2^{a, b}$ with the
$\di$-arrow $\pr_2^{a, b} = \pr_b$ from Example~\ref{ex: exds2}.
\end{example}

In a $(\di, \Sigma)$-category $\C C$ with a terminal object $1$ we can recover the standard equality
$z = \big(\pr_1(z), \pr_2(z)\big)$, where $z$ is an element of the Sigma-type (set) in $\MLTT$ $(\BST)$.
If $i \in a$, $\lambda \in \fHom(a)$ and $\Phi \in \dHom(a, \lambda)$, then $\Phi(i) \in \dHom(1, \lambda(i))$
\vspace{-4mm}
\begin{center}
\begin{tikzpicture}

\node (E) at (0,0) {$1$};
\node[right=of E] (F) {$a$};
\node[right=of F] (A) {};

\draw[>->] (E)--(F) node [midway,above] {$i$};
\draw[MyBlue,->] (F)--(A) node [midway,above] {$\lambda$};
\draw[MyBlue,->,bend right=40] (E) to node [midway,below] {$\lambda(i)$} (A) ;

\end{tikzpicture}
\end{center}
\vspace{-4mm}
and by Definition~\ref{def: dcat} we have that
\begin{equation}\label{eq: prelement}
\pr_2^{a, \lambda}\big(\Sigma_{\lambda}i\big) = \pr_2^{1, \lambda(i)}
\end{equation}
\begin{center}
\resizebox{5cm}{!}{%
\begin{tikzpicture}

\node (E) at (0,0) {$\sum_1 \lambda(i) $};
\node[right=of E] (K) {};
\node[right=of K] (F) {$\sum_a \lambda$};
\node[below=of E] (A) {$1 $};
\node [below=of F] (D) {$a$.};

\draw[>->] (E) to node [midway,above] {$\Sigma_{\lambda}i$} (F) ;
\draw[>->] (A) to node [midway,below] {$i$} (D) ;
\draw[->] (E)--(A) node [midway,left] {$!$};
\draw[->] (F)--(D) node [midway,right] {$\pr_1^{a, \lambda}$};

\end{tikzpicture}
}
\end{center}

\begin{proposition}\label{prp: elsigma}
Let $\C C$ be a $(\di, \Sigma)$-category with a terminal object $1$, and 
$w, z \in \sum_a \lambda$, hence $\pr_2^{a, \lambda}(z) \in \dHom\big(1, (\lambda \circ \pr_1^{a, \lambda})(z)\big)$,
\begin{center}
\resizebox{4cm}{!}{%
\begin{tikzpicture}

\node (E) at (0,0) {$1$};
\node[right=of E] (F) {$\sum_a \lambda$};
\node[right=of F] (A) {$a$};
\node[right=of A] (B) {.};

\draw[>->] (E)--(F) node [midway,below] {$z$};
\draw[->] (F)--(A) node [midway,below] {$\pr_1^{a, \lambda} $};
\draw[MyBlue,->] (A)--(B) node [midway,above] {$\lambda$};
\draw[MyBlue,->,bend left=80] (E) to node [midway,above] {$\lambda\big(\pr_1^{a, \lambda}(z)\big)$} (B) ;
\draw[>->,bend left=40] (E) to node [midway,above] {$\pr_1^{a, \lambda}(z)$} (A) ;

\end{tikzpicture}
}
\end{center}
\normalfont (i) 
\itshape 
If $u \in \sum_1 \lambda\big(\pr_1^{a, \lambda}(z)\big)$ is the unique global element of 
$\sum_1 \lambda\big(\pr_1^{a, \lambda}(z)\big)$ determined by the following pullback
 \begin{center}
 \resizebox{6cm}{!}{%
\begin{tikzpicture}

\node (E) at (0,0) {$\mathsmaller{\sum_1 \lambda\big(\pr_1^{a, \lambda}(z)\big)} $};
\node[right=of E] (K) {};
\node[right=of K] (F) {$\sum_a \lambda$};
\node[below=of E] (A) {$1 $};
\node [below=of F] (D) {$a$.};
\node[left=of E] (G) {};
\node[above=of G] (H) {$1$};

\draw[>->] (E) to node [midway,above] {$\mathsmaller{\Sigma_{\lambda}\pr_1^{a, \lambda}(z)}$} (F) ;
\draw[>->] (A) to node [midway,below] {$\pr_1^{a, \lambda}(z)$} (D) ;
\draw[->] (E)--(A) node [midway,left] {$!$};
\draw[->] (F)--(D) node [midway,right] {$\pr_1^{a, \lambda}$};
\draw[->, bend right=50] (H) to node [midway,left] {$1_1$} (A);
\draw[->, bend left=30] (H) to node [midway,above] {$ \ z$} (F);
\draw[MyRed, ->,dashed] (H) to node [midway,above] {$ \ u$} (E) ;
\draw[->] (E)--(A) node [midway,left] {$!$};

\end{tikzpicture}
}
\end{center}
the following equations hold:
\begin{equation}\label{eq: pr0}
! \circ u = 1_1,
\end{equation}
\begin{equation}\label{eq: pr1}
z = \big(\Sigma_{\lambda}\pr_1^{a, \lambda}(z)\big) \circ u,
\end{equation}
\begin{equation}\label{eq: pr2}
\pr_2^{a, \lambda}(z) = \pr_2^{1, \lambda\big(\pr_1^{a, \lambda}(z)\big)}(u)
\end{equation}
\normalfont (ii) 
\itshape If $u{'} \in \sum_1 \lambda\big(\pr_1^{a, \lambda}(w)\big)$ is the unique global element of 
$\sum_1 \lambda\big(\pr_1^{a, \lambda}(w)\big)$ determined by the corresponding pullback for $w$, and $\lambda_{ij}$ is 
the corresponding transport arrow,
 \begin{center}
 \resizebox{6cm}{!}{%
\begin{tikzpicture}

\node (E) at (0,0) {$\sum_1 \lambda(j) $};
\node[right=of E] (K) {};
\node[right=of K] (F) {$\sum_a \lambda$};
\node[below=of E] (A) {$1 $};
\node [below=of F] (D) {$a$.};
\node[left=of E] (G) {};
\node[above=of G] (H) {$\sum_1 \lambda(i)$};

\draw[>->] (E) to node [midway,above] {$\Sigma_{\lambda}j$} (F) ;
\draw[>->] (A) to node [midway,below] {$j$} (D) ;
\draw[>->] (A) to node [midway,above] {$i$} (D) ;
\draw[->] (E)--(A) node [midway,left] {$!$};
\draw[->] (F)--(D) node [midway,right] {$\pr_1^{a, \lambda}$};
\draw[->, bend right=50] (H) to node [midway,left] {$!$} (A);
\draw[->, bend left=30] (H) to node [midway,above] {$ \  \Sigma_{\lambda}i$} (F);
\draw[MyRed, ->,dashed] (H) to node [midway,above] {$ \  \lambda_{ij}$} (E) ;
\draw[->] (E)--(A) node [midway,left] {$!!$};

\end{tikzpicture}
}
\end{center}
where $i := \pr_1^{a, \lambda}(z)$ and $j := \pr_1^{a, \lambda}(w)$, then 
\begin{equation}\label{eq: pr3}
z = w \TOT \pr_1^{a, \lambda}(z) = \pr_1^{a, \lambda}(w) \ \&  \ u{'} = \lambda_{ij} \circ u,
\end{equation}
\begin{equation}\label{eq: pr4}
z = w \To \pr_2^{a, \lambda}(z) = \pr_2^{a, \lambda}(w).
\end{equation}
\end{proposition}

\begin{proof}
 (i) Equations~(\ref{eq: pr0}) and~(\ref{eq: pr1}) correspond to the commutative triangles of the above diagram.
 By equation~(\ref{eq: pr0}) we
 get
 \begin{align*}
 \pr_2^{a, \lambda}(z) & = \pr_2^{a, \lambda}\bigg(\big(\Sigma_{\lambda}\pr_1^{a, \lambda}(z)\big) \circ u\bigg)\\
 & =  \bigg[\pr_2^{a, \lambda}\big(\Sigma_{\lambda}\pr_1^{a, \lambda}(z)\big)\bigg](u)\\
 & \stackrel{((\ref{eq: prelement})} =  \pr_2^{1, \lambda\big(\pr_1^{a, \lambda}(z)\big)}(u).
\end{align*}
(ii) If $z = w$, the  equality $\pr_1^{a, \lambda}(z) = \pr_1^{a, \lambda}(w)$ follows immediately. By equation~(\ref{eq: pr2})
we get
$$\big(\Sigma_{\lambda}j \circ  \lambda_{ij}\big)  \circ u  = \big(\Sigma_{\lambda}i\big) \circ u = z = w = \Sigma_{\lambda}j \circ u{'},$$
and as $\Sigma_{\lambda}j$ is a mono, we get $ \lambda_{ij} \circ u  = u{'}$. For the converse implication we have 
that\footnote{The equality $z = w$ relies only on the equality $u{'} = \lambda_{ij} \circ u $.}
$$w = \Sigma_{\lambda}j \circ u{'} =  \Sigma_{\lambda}j \circ  \lambda_{ij} \circ u = \Sigma_{\lambda}i \circ u = z.$$
Moreover, we have that
\begin{align*}
 \pr_2^{a, \lambda}(w) & \stackrel{((\ref{eq: pr2})} = \pr_2^{1, \lambda(j)}(u{'})\\
 & \stackrel{((\ref{eq: pr3})} = \pr_2^{1, \lambda(j)}(\lambda_{ij} \circ u)\\
 & = \big[\pr_2^{1, \lambda(j)}(\lambda_{ij})\big](u)\\
 & \stackrel{((\ref{eq: prelement})} = \big[\big[\pr_2^{a, \lambda}\big(\Sigma_{\lambda}j\big)\big](\lambda_{ij})\big](u)\\
 & = \big[\pr_2^{a, \lambda}\big(\Sigma_{\lambda}j \circ \lambda_{ij}\big)\big](u)\\
 & = \big[\pr_2^{a, \lambda}\big(\Sigma_{\lambda}i\big)\big](u)\\
 & \stackrel{((\ref{eq: prelement})} = \pr_2^{1, \lambda(i)}(u)\\
 & \stackrel{((\ref{eq: pr2})} = \pr_2^{a, \lambda}(z). \qedhere
\end{align*}
\end{proof}

Equations~(\ref{eq: pr1}) and~(\ref{eq: pr2})  
are the category-theoretic version of the equality $z = \big(\pr_1(z), \pr_2(z)\big)$ and
the equivalence~(\ref{eq: pr3}) is the category-theoretic version of the canonical equality 
on the Sigma-set in $\BST$ given in Example~\ref{ex: ssets}.

\section{Concluding comments}
\label{sec: concl}

To the arrow-structure $C_1$ of a category $\C C$ a family-structure $C_2$ was added implicitly already in the
definition of a category with attributes or of a type-category. Sigma-objects were studied in the ``two-dimensional world''
of $C_1$ and $C_2$ and dependency was defined for every object $a$, for every $\lambda \in \fHom(a)$, and every $\mu \in 
\fHom(\sum_a \lambda)$, as an appropriate family $\prod_{\lambda}\mu \in \fHom(a)$ (see~\cite{Pi00}, pp.~120-121.)
Here we added a third (independent) dimension $C_3$ of dependent arrows 
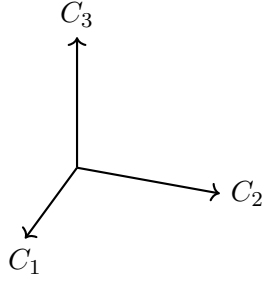
\tdplotsetmaincoords{60}{110}
\begin{center}
\begin{tikzpicture}[scale=2,tdplot_main_coords]
  
  \def\rvec{.8}
  \def\thetavec{30}
  \def\phivec{60}
  
  \coordinate (O) at (0,0,0);
  \draw[thick,->] (0,0,0) -- (1,0,0) node[below=]{$C_1$};
  \draw[thick,->] (0,0,0) -- (0,1,0) node[right=]{$C_2$};
  \draw[thick,->] (0,0,0) -- (0,0,1) node[above=]{$C_3$};
  
  

\end{tikzpicture}
\end{center}
that allowed us to express dependency through the third ``dimension'' $C_3$ alone, and independently from
Sigma-objects. Actually, the definition of Sigma-objects in the ``three-dimensional world'' of $C_1, C_2$ and $C_3$
incorporated the second-projection-dependent arrow providing a closer analogy to the study of the Sigma-type (set) in
$\MLTT$ $(\BST)$. The categorical formulation of dependency is fundamental in $\di$-categories, 
while it is very complicated and dependent to Sigma-objects in type-categories.

The importance of $\di$-categories also lies on the possibility of having  dependent arrows that are not generated from the 
Sigma-objects as the dependent objects, or the global sections, according to Theorem~\ref{thm: typeisdi}. 
We have dependent arrows ``before'' and independently
from the Sigma-construction, as in the categories of Examples~\ref{ex: dtrivial} and~\ref{ex: altdsets}.
These categories reflect the fundamental character of dependent arrows.

As it is noted in~\cite{BW12}, p.~331, ``one way of constructing fibrations is by the Grothendieck construction $\ldots$,
which is a generalisation of the semidirect product construction for monoids''. This function of the Grothendieck construction 
is a special case of the general function of a Sigma-object. Various categories of Sigma-objects can be defined in the framework of
$(\f, \Sigma)$-categories, or $(\di, \Sigma)$-categories, in which the first projection becomes a split 
fibration. 
This reinforces the choice of the Sigma-notation for Grothendieck categories and at 
the same time 
explains why the theory of $(\f, \Sigma)$-categories, or of $(\di, \Sigma)$-categories,
can be seen as a generalisation of the theory of Grothendieck categories. 
More relations and connections to fibrations need to be explored, as the defining clauses of a splitting cleavage for 
a Grothendieck fibration are clearly very similar to the strictness 
condition $(\si_1)$ and $(\si_2)$.

As it is mentioned in the Introduction, a family of sets can be described as a fibration instead of 
using a pointwise indexing. As family-arrows generalise the pointwise indexed families of sets, the 
cofamily-arrows generalise the families of sets determined by fibrations. Dually to family-arrows, a \textit{cofamily-arrow}
has a fixed 
codomain $b$ and composes in a coherent way with the arrows of $\C C$ with domain $b$:\\[1mm]
$(\cf_1)$ $ \ 1_b \circ p = p$
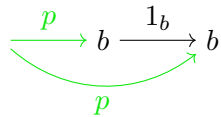
\begin{center}
\begin{tikzpicture}

\node (E) at (0,0) {};
\node[right=of E] (F) {$b$};
\node[right=of F] (A) {$b$};

\draw[->] (F)--(A) node [midway,above] {$1_b$};
\draw[MyGreen,->] (E)--(F) node [midway,above] {$p$};
\draw[MyGreen,->,bend right=40] (E) to node [midway,below] {$p$} (A) ;

\end{tikzpicture}
\end{center}
$(\cf_2)$ $ \ (g \circ f) \circ p =  g \circ (f \circ p)$
\begin{center}
\resizebox{4cm}{!}
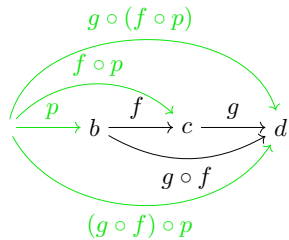
{%
\begin{tikzpicture}

\node (E) at (0,0) {};
\node[right=of E] (F) {$b$};
\node[right=of F] (A) {$c$};
\node[right=of A] (B) {$d$};

\draw[->] (A)--(B) node [midway,above] {$g$};
\draw[->] (F)--(A) node [midway,above] {$f \ $};
\draw[MyGreen,->] (E)--(F) node [midway,above] {$ \ p$};
\draw[->,bend right] (F) to node [midway,below] {$g \circ f$} (B) ;
\draw[MyGreen,->,bend right=60] (E) to node [midway,below] {$(g \circ f) \circ p$} (B) ;
\draw[MyGreen,->,bend left=45] (E) to node [midway,above] {$ \ f \circ p$} (A) ;
\draw[MyGreen,->,bend left=75] (E) to node [midway,above] {$g \circ (f \circ p)$} (B) ;

\end{tikzpicture}
}
\end{center}
Consequently, all notions presented here have a dual counterpart. The study of 
\textit{coSigma-objects} and \textit{codependent 
arrows} is a necessary complement to the development of a form of \textit{Dependent Category Theory} that was only 
started here.

\noindent
\textbf{Acknowledgement}\\[1mm]
I would like to thank Benno van den Berg for pointing~\cite{Pi00} to me during a discussion we had at an early stage of this work.


\begin{thebibliography}{10}\label{bibliography}


\bibitem{AR10} P.~Aczel, M. Rathjen: \textit{Constructive Set Theory}, book draft, 2010.
\bibitem{Aw10} S.~Awodey: \textit{Category Theory}, Oxford University Press, 2010.
%
\bibitem{BW12} M.~Barr, C.~Wells: \textit{Category Theory for Computing Science}, Reprints in Theory and Applications in 
Category Theory, Center de Recherche Math\'ematique, Universit\'e de Montr\'eal, 2012.

\bibitem{Be85} J.~B\'enabou: Fibered categories and the foundations of naive category theory, The Journal of Symbolic
Logic, Volume 50, Number 1, 1985, 10--37.

\bibitem{Bi67} E. Bishop: \textit{Foundations of Constructive Analysis}, McGraw-Hill, 1967.
\bibitem{BB85} E. Bishop and D. S. Bridges: \textit{Constructive Analysis}, Grundlehren der math. Wissenschaften 279,
Springer-Verlag, Heidelberg-Berlin-New York, 1985.

\bibitem{Cu93} P.-L.~Curien: Substitution up to isomorphism, Fundamenta Informaticae, 19, 1993, 51--85.

\bibitem{Ca78} J.~Cartmell: \textit{Generalised algebraic theories and contextual categories}, DPhil,~Thesis, Oxford, 1978. 
\bibitem{Ca86} J.~Cartmell: Generalised algebraic theories and contextual categories, Annals of Pure and Applied Logic,
32, 1986, 209--243.

\bibitem{Dy96} P.~Dybjer: Internal Type Theory, in S.~Berardi, M.~Coppo (Eds.) \textit{Types for Proofs and Programs}, 
LNCS, Vol.~1158, 1996, 120--134.


\bibitem{Eh88} T.~Ehrhard: A Categorical Semantics of Constructions, LICS 1988, IEEE Computer Society Press, 1988, 264--273.

%

\bibitem{Ho97} M.~Hofmann: Syntax and semantics of dependent types, in A.~Pitts, P.~Dybjer (Eds.) \textit{Semantics and Logics of Computation},
Cambridge University Press, 1997, 79--130.

\bibitem{HS98} M.~Hofmann, T.~Streicher: The groupoid interpretation of type theory, in~\cite{SS98}, 1998, 
83--111.

\bibitem{Ja99} B.~Jacobs: \textit{Categorical Logic and Type Theory}, Elsevier, 1999. 

\bibitem{JP78} P.~T.~Johnstone, R.~Par\'e (Eds): \textit{Indexed Categories and Their Applications}, Springer-Verlag, 1978. 



\bibitem{LS86} J.~Lambek, P.~J.~Scott: \textit{Introduction to higher order categorical logic}, Cambridge University Press, 1986.

\bibitem{MM92} S.~Mac Lane, I.~Moerdijk: \textit{Sheaves in Geometry and Logic}, Springer-Verlag, 1992.

\bibitem{ML75} P. Martin-L\"{o}f: An intuitionistic theory of types: predicative part, in H. E. Rose and
J. C. Shepherdson (Eds.) \textit{Logic Colloquium'73}, pp.73-118, North-Holland, 1975.
\bibitem{ML84} P. Martin-L\"{o}f: \textit{Intuitionistic type theory: Notes by Giovanni Sambin on a series
of lectures given in Padua, June 1980}, Napoli: Bibliopolis, 1984.
\bibitem{ML98} P.~Martin-L\"{o}f: An intuitionistic theory of types, in~\cite{SS98}, 127--172. 



%
\bibitem{MRR88} R.~Mines, F.~Richman, W.~Ruitenburg: \textit{A course in constructive algebra}, Springer, 1988.

\bibitem{My75} J.~Myhill: Constructive Set Theory, J.~Symbolic Logic 40, 1975, 347--382.


\bibitem{Pa12a} E.~Palmgren: Proof-relevance of families of setoids and identity in type theory, Arch. Math. Logic, 51,
2012, 35--47. 

\bibitem{Pa12b} E.~Palmgren: Constructivist and structuralist foundations: Bishop's and Lawvere's theories
of sets, Annals of Pure and Applied Logic 163, 2012, 1384--1399.

\bibitem{Pa16} E. Palmgren: The Grothendieck construction and models for dependent types, preprint, 2016.

\bibitem{Pe19} I.~Petrakis: Dependent Sums and Dependent Products in Bishop's Set Theory, in P. Dybjer et. al. (Eds) 
TYPES 2018, LIPIcs, Vol. 130, Article No. 3, 2019.
\bibitem{Pe20} I. Petrakis: \textit{Families of Sets in Bishop Set Theory}, Habilitation Thesis, LMU, 2020.
\bibitem{Pe21b} I.~Petrakis:  From the Sigma-type to the Grothendieck construction,
arXiv:2109.04239v1, 2021. 
%
%
\bibitem{Pe22a} I.~Petrakis: The distributivity of the category of dependent objects over the Groethendieck category,
in Coquand et. al (Eds) \textit{Geometric Logic, Constructivisation, and Automated Theorem Proving}, Dagstuhl 
Seminar 21472, Schloss Dagstuhl, Leibniz-Zentrum fuer Informatik, 2022, p.~166, DOI: 10.4230/DagRep.11.10.151




\bibitem{Pe22b} I.~Petrakis: Proof-relevance in Bishop-style constructive mathematics, Mathematical Structures in Computer
Science, Volume 32, Issue 1, 1--43.



\bibitem{Pi00} A.~M.~Pitts: Categorical logic, in S.~Abramsky, D.~M.~Gabbay, T.~S.~E.~Maibaum (Eds.)
\textit{Handbook of Logic in Computer Science}, Vol.~5, Clarendon Press, Oxford, 2000, 39--128.

\bibitem{Ri16} E.~Riehl: \textit{Category Theory in Context}, Dover Publications Inc., 2016.


\bibitem{SS98} G.~Sambin, J.~M.~Smith (Eds.): \textit{Twenty-five years of constructive type theory},
Oxford University Press, 1998.

\bibitem{Se84} R.~A.~G.~Seely: Locally cartesian closed categories and type theories, Mathematical Proceedings of the
Cambridge Philosophical Society, 95, 1984, 33--48.
%
%

 
 
\end{thebibliography}
\end{document}